\documentclass[10pt,letterpaper,twoside]{article}
\usepackage[margin=1in, marginparwidth=17mm]{geometry}
\usepackage{amsmath,amsthm,amssymb}
\usepackage{mathtools}
\usepackage[utf8]{inputenc}
\usepackage[T1]{fontenc}
\usepackage{amsfonts}
\usepackage{graphicx}
\usepackage{lmodern}
\usepackage{enumerate}
\usepackage{hyperref}
\usepackage{booktabs}
\usepackage{float}
\usepackage{fancyhdr}
\usepackage{bold-extra}
\usepackage{tabularx}
\usepackage{mathrsfs}
\usepackage{bm} 
\usepackage[normalem]{ulem} 
\usepackage[dvipsnames]{xcolor}
\usepackage{extarrows} 
\usepackage{empheq}
\usepackage{calligra}
\usepackage[cmtip,all]{xy}

\def\abs#1{\left| #1 \right|}
\newcommand{\norm}[1]{\left\lVert#1\right\rVert}

\newcommand{\inparen}[1]{\left(#1\right)}             
\newcommand{\inbraces}[1]{\left\{#1\right\}}           
\newcommand{\insquare}[1]{\left[#1\right]}             
\newcommand{\inangle}[1]{\left\langle#1\right\rangle} 

\let\tilde\widetilde
\let\bar\overline

\newcommand{\parag}[1]{ \paragraph{\bf #1}}

\newcommand{\ud}{\mathrm{d}}

\DeclareMathOperator{\Var}{\rm Var}

\DeclareMathOperator{\Cov}{\rm Cov}

\newcommand{\diag}{\textnormal{diag}}

\newcommand{\Tr}{\mathop{\mathrm{Tr}}}
\newcommand{\Hess}{\mathop{\mathrm{Hess}}}

\def\vec#1{{\boldsymbol{#1}}}

\DeclarePairedDelimiterX{\expectarg}[1]{[}{]}{%
  \ifnum\currentgrouptype=16 \else\begingroup\fi
  \activatebar#1
  \ifnum\currentgrouptype=16 \else\endgroup\fi
}

\newcommand{\innermid}{\nonscript\;\delimsize\vert\nonscript\;}
\newcommand{\activatebar}{%
  \begingroup\lccode`\~=`\|
  \lowercase{\endgroup\let~}\innermid 
  \mathcode`|=\string"8000
}

\newcommand{\longsquiggly}{\xymatrix{{}\ar@{~>}[r]&{}}}

\newcommand{\bx}{\bm{x}}
\newcommand{\by}{\bm{y}}

\newcommand{\Eb}{\mathbf{E}}

\newcommand{\cA}{\mathcal{A}}

\newcommand{\cF}{\mathcal{F}}
\newcommand{\cG}{\mathcal{G}}

\newcommand{\cL}{\mathcal{L}}

\newcommand{\cN}{\mathcal{N}}

\newcommand{\btheta}{\bm{\theta}}

\newcommand{\blambda}{\bm{\lambda}}

\newcommand{\bxi}{\bm{\xi}}

\newcommand{\bSigma}{\bm{\Sigma}}

\newcommand{\LL}{\mathbb{L}}
\newcommand{\MM}{\mathbb{M}}

\newcommand{\PP}{\mathbb{P}}
\newcommand{\QQ}{\mathbb{Q}}
\newcommand{\RR}{\mathbb{R}}
\newcommand{\TT}{\mathbb{T}}

\usepackage{algorithm}
\usepackage{algpseudocode}
\usepackage{algorithmicx}

\algnewcommand\algorithmicphase{\textbf{Phase}}
\algnewcommand\PHASE{\item[\algorithmicphase]}


\DeclarePairedDelimiterX{\infdivx}[2]{(}{)}{%
  #1\;\delimsize\|\;#2%
}

\newtheorem{theorem}{Theorem}[section]
\newtheorem{corollary}[theorem]{Corollary}
\newtheorem{lemma}[theorem]{Lemma}

\newtheorem{proposition}[theorem]{Proposition}

\theoremstyle{definition}
\newtheorem{definition}{Definition}[section]
\newtheorem{example}{Example}[section]

\theoremstyle{remark}
\newtheorem*{remark}{Remark}

\newenvironment{proofof}[1]{\begin{trivlist} \item {\textbf{{Proof of
#1}.~~}}}
  {\qed\end{trivlist}}

\makeatletter
\def\blfootnote{\xdef\@thefnmark{}\@footnotetext}
\makeatother

\usepackage{pgf,tikz,pgfplots}
\pgfplotsset{compat=1.15}
\usepackage{mathrsfs}
\usetikzlibrary{arrows, decorations.pathreplacing, calligraphy}
\definecolor{ffffff}{rgb}{1,1,1}

\newcommand{\footremember}[2]{%
   \footnote{#2}
    \newcounter{#1}
    \setcounter{#1}{\value{footnote}}%
}
\newcommand{\footrecall}[1]{%
    \footnotemark[\value{#1}]%
}

\usepackage{tikz-cd}

\begin{document}

\title{Random projections beyond zero overlap}
\author{Timothy L.H. Wee\footremember{fn:author}{Department of Statistics and Data Science, Yale University. Email: timothy.wee@yale.edu, sekhar.tatikonda@yale.edu } \and Sekhar Tatikonda\footrecall{fn:author} }
\date{\today}

\maketitle

\begin{abstract}
    A random vector whose norm and overlap (inner product with an independent copy) concentrates is shown to have random low-dimensional projections that are approximately random Gaussians. Conversely, asymptotically random Gaussian projections imply these hypotheses. This extends and unites several existing results in geometric functional analysis and spin glasses. Applications include a large-system characterization of the joint law of cavity fields in the Sherrington-Kirkpatrick model.\blfootnote{\emph{{MSC2020 subject classifications.} 60K35, 60F05, 82B44}}  
\end{abstract}

\thispagestyle{empty}


\section{Introduction}

Properties of low-dimensional projections of high-dimensional distributions are of fundamental importance in various fields of probability, statistics, computer science, and engineering. A non-exhaustive list of topics includes central limit theorems and large deviation principles for projections of high-dimensional bodies with geometric structure in geometric functional analysis \cite{meckes2012projections_3}, \cite{reeves2017conditional}, \cite{klartag2007central}, \cite{anttila2003central}, \cite{kim2022asymptotic}, \cite{gantert2017large}, projection pursuit \cite{diaconis1984asymptotics}, \cite{dumbgen2013low}, \cite{dasgupta2012concentration}, \cite{sudakov1978typical}, \cite{von1997sudakov}, and limit theorems for cavity/local fields and Thouless-Anderson-Palmer equations in spin glasses \cite{chen2010central}, \cite{chatterjee2010spin}, \cite[Sections 1.5-1.7]{talagrand2010mean}, \cite{chen2022tap}.

To set the stage, let $\vec{x}$ be a random vector in $\RR^N$, and let $\Theta = (\btheta_1,\dots, \btheta_k)$ be an $N \times k$ matrix of \emph{projection directions} so that $\Theta^\top \vec{x} \in \RR^k$ is the projection of $\vec{x}$ onto a $k$-dimensional subspace, $k < N$. For instance, we can take $\Theta$ to be a matrix whose columns constitute a random orthonormal basis for a $k$-dimensional subspace in $\RR^N$. Alternatively, if $k$ is not growing too fast with with $N$, this is approximately equivalent to taking the columns of $\Theta$ to be independent and identically distributed (iid) Gaussian vectors with zero mean and covariance $N^{-1} I$. In this paper we will assume the Gaussian projections setting.

Our main result shows that if $\vec{x}$ satisfies the \emph{thin-shell} and \emph{overlap concentration} hypotheses \eqref{eq:opening_TSOC_hypotheses}, then the distribution of the projection $\Theta^\top \vec{x}$ will be close to a \emph{random} Gaussian distribution for large $N$ and for most projection directions $\Theta$. Moreover, a partial converse is obtained, in the sense that if the asymptotic distribution of $\Theta^\top \vec{x}$ is indeed this random Gaussian, then the thin-shell and overlap concentration hypotheses are also true. 

More specifically, let $\vec{x}^1, \vec{x}^2$ be independent copies of $\vec{x}$. If for some constants $0 \leq q < \rho$ it holds that
\begin{align}
    \frac{1}{N} \norm{\vec{x}}^2 \simeq \rho, \quad \textnormal{and} \quad\quad \frac{1}{N} \vec{x}^1\boldsymbol{\cdot} \vec{x}^2 \simeq q, \quad \textnormal{both in } L_2,
    \label{eq:opening_TSOC_hypotheses}
\end{align}
then with $\vec{z}$, $\bxi$ denoting independent standard Gaussian random vectors in $\RR^k$, 
\begin{align}
    \cL\!\inparen{\Theta^\top \vec{x} \, | \, \Theta} \simeq \cL\inparen{ \sqrt{q} \vec{z} + \sqrt{\rho - q} \bxi \, | \, \vec{z}  }, \quad \textnormal{w.h.p.~$\Theta$ and $\vec{z}$},
    \label{eq:opening_mainResult_asympForm}
\end{align}
where $\simeq$ is closeness over integrals of some class of test functions. Moreover, if \eqref{eq:opening_mainResult_asympForm} holds, then \eqref{eq:opening_TSOC_hypotheses} is true asymptotically, where $\simeq$ holds ``in probability''.

Furthermore, we also obtain a useful `intermediate form' for the approximate projection distribution:
\begin{align}
    \cL\!\inparen{\Theta^\top \vec{x} \, | \, \Theta} \simeq \cL\inparen{  \Theta^\top\! \inangle{\vec{x}} + \sqrt{\rho - q} \bxi \, | \, \Theta  }, \quad \textnormal{w.h.p.~$\Theta$},
    \label{eq:opening_mainResult_partialLimitForm}
\end{align}
where $\inangle{\vec{x}}$ is the mean vector of $\vec{x}$.

The norm concentration in \eqref{eq:opening_TSOC_hypotheses}, also called thin-shell, is a fundamental notion in random projections. For instance, under symmetry conditions on the distribution of $\vec{x}$ (for instance isotropy---where $\Cov \vec{x} = I_N$), it is known that thin-shell characterizes the low-dimensional Gaussian behavior \cite{anttila2003central}, \cite{klartag2007central}. 

On the other hand, the overlap concentration in the second equation of \eqref{eq:opening_TSOC_hypotheses} is perhaps less-understood. Importantly, it can be shown that the non-zero overlap concentration generalizes the previously considered settings of isotropy and bounded covariance eigenvalues (see Section \ref{sec:further_background}, and for papers in this setting see \cite{meckes2012projections_3}, \cite{bobkov2010theorem}). The existing projection results concerning overlap either (1) assume it is zero ($q = 0$ in \eqref{eq:opening_TSOC_hypotheses}) \cite{diaconis1984asymptotics}, \cite{reeves2017conditional}, \cite{dumbgen2013low}, or (2) allow non-zero overlap concentration but are restricted to specific spin glass measures and are limited to one-dimensional projections ($k = 1$) \cite{chen2010central}, \cite{chatterjee2010spin}, \cite[Theorem 1.7.11]{talagrand2010mean}. Our main result is informally:
\begin{align*}
    &\textnormal{\emph{We extend and unite the existing projection results}} \\
    &\textnormal{ \emph{in the spin glass and random projections literature to simultaneously} } \\
    &\textnormal{ \emph{allow for non-zero overlaps as well as multi-dimensional } } \\
    &\textnormal{ \emph{and distribution-free projections with non-asymptotic rates.} }
\end{align*}

The results in this paper can be construed as a general answer to \cite[Research Problem 1.7.12]{talagrand2010mean} (which was already answered specifically for the Sherringon-Kirkpatrick model in \cite{chen2010central}).

Another insight provided by this paper is that in the proof, we extend a technique introduced by Hoeffding \cite[Theorem 3.3]{hoeffding1952large}, and further developed in \cite[Lemma 4.1]{dumbgen2013low}, that characterizes the weak convergence of conditional distributions and relates it to the weak convergence of appropriately replicated unconditional distributions. More precisely, they showed that for some probability measure $\QQ$,
\begin{align}
    \cL(\Theta^\top \vec{x}^1, \Theta^\top \vec{x}^2) \longsquiggly \QQ \otimes \QQ \quad \Longleftrightarrow \quad \cL(\Theta^\top \vec{x} \, | \, \Theta) \longsquiggly \QQ \;\; \textnormal{w.h.p.~$\Theta$},
    \label{eq:opening_Hoeffding}
\end{align}
where `$\longsquiggly$' is weak convergence. In the non-zero overlap setting, the LHS fails as the joint distribution does not converge to a product measure. However, this is overcome by implementing a `conditional' version of the Hoeffding technique. Additionally, it is shown that this technique can produce non-asymptotic rates.

\parag{Organization.} The rest of the introduction is organized as follows: we state in Section \ref{sec:main_results} the main results; we outline the main proof ideas in Section \ref{sec:outline}; and we give additional background in Section \ref{sec:further_background}.

In Section \ref{sec:proof_main_result} we prove the main result Theorem \ref{thm:ProjResult_DisorderedCase_supOverBL_LM_2p}; and in Section \ref{sec:examples} some applications are given; and finally in Section \ref{sec:partial_converse} we prove the partial converse Proposition \ref{prop:projResult_converse}.

\subsection{Main results}
\label{sec:main_results}

The high-dimensional objects to be projected are random vectors $\vec{x}$ in $\RR^N$ with expectations denoted by $\inangle{\cdot}$. The random projection directions are $N \times k$ matrices $\Theta$ with iid entries of zero-mean Gaussians with variance $1/N$. The random vector $\vec{x}$ and projection direction $\Theta$ are always assumed to be independent. The projection dimension $k < N$ is a positive integer that is fixed and not growing with $N$. In what follows $\Eb_{Y}$ is used to mean that the expectation is taken with respect to the random element $Y$ only.


Our main result is as follows.

\begin{theorem}
Let $0 \leq q < \rho$ be constants such that
\begin{align}
    \inangle{ \inparen{\frac{\norm{\vec{x}}^2}{N} - \rho}^{2}  } \leq c_1, \quad \inangle{ \inparen{\frac{\vec{x}^1 \boldsymbol{\cdot} \vec{x}^2}{N} - q}^{2}  } \leq c_2,
    \label{eq:projResult_TSOC_hypotheses}
\end{align}
for some numbers $c_1$, $c_2$ that may depend on $N$. Let $\bxi$ be a standard Gaussian vector in $\RR^k$, independent of all other sources of randomness. Then for every integer $p \geq 1$, for constants $K_1, K_2 > 0$ that depend on $p, \rho, q$, and for $d_1, d_2$ defined below in \eqref{eq:d1_d2_rates}, the following happens.
\begin{enumerate}
    \item We have
    \begin{align}
        &\sup_{\substack{\norm{g}_{\textnormal{Lip}} \leq L \\ \norm{g}_{\infty} \leq M} } \Eb_{\Theta}  \insquare{ \inparen{ \inangle{g(\Theta^\top \vec{x})  } - \Eb_{\bxi} \insquare{  g\inparen{\Theta^\top\!\inangle{\vec{x}} + \sqrt{\rho - q}\bxi  }  }  }^{2p}  } \nonumber\\
        &\quad\quad\quad\quad\quad\quad\quad\quad\quad\quad\quad\quad\quad \leq \frac{K_1 k L M^{2p - 1} }{N - 1} \inparen{ d_1(c_1) +  d_2(c_2) \boldsymbol{1}_{q > 0} + N\inparen{c_2}^{1/4}\boldsymbol{1}_{q = 0}  }.
        \label{eq:projResult_partialAsymptotic}
    \end{align}
    \item Furthermore, let $\vec{z} \in \RR^k$ be a standard Gaussian vector, independent of all other sources of randomness, then 
    \begin{align}
        &\sup_{\substack{\norm{g}_{\textnormal{Lip}} \leq L \\ \norm{g}_{\infty} \leq M} } \Eb_{\Theta}\Eb_{\vec{z}}  \insquare{ \inparen{ \inangle{g(\Theta^\top \vec{x})  } - \Eb_{\bxi} \insquare{  g\inparen{\sqrt{q}\vec{z} + \sqrt{\rho - q}\bxi  }  }  }^{2p}  }  \leq \frac{K_2 k L M^{2p - 1} }{N - 1} \inparen{ d_1(c_1) + d_2(c_2)  }.
        \label{eq:projResult_asymptotic}
    \end{align}
\end{enumerate}
Here, $d_1, d_2 : \RR_+ \rightarrow \RR_+$ are defined by
\begin{align}
    d_1(y) := \sqrt{3N^2 y + 4N\rho \sqrt{y} + 2N\rho^2 }, \quad 
    d_2(y) := \sqrt{3N^2 y + 4Nq\sqrt{y} + 2Nq^2 }.
    \label{eq:d1_d2_rates}
\end{align}
\label{thm:ProjResult_DisorderedCase_supOverBL_LM_2p}
\end{theorem}

The above non-asymptotic rates can be turned into asymptotic statements about the convergence of random probability measures. To state the result, we first explain some terminology.
\begin{definition}
    Let $\inparen{\Omega, \cF, \PP}$ be a probability space. Say that $\inparen{\mu_{N}}_{N \geq 1}$ is a sequence of \emph{random probability measures} on a Polish space $\Sigma$ if $\mu_{N} = \mu_{N,(\cdot)}$ defines a mapping from $\Omega$ into $M_1(\Sigma)$, where $M_1(\Sigma)$ is the space of Borel probability measures on $\Sigma$, equipped with the topology of weak convergence.
\end{definition}

Define the following notion of weak convergence for random measures.
\begin{definition}
    Let $\inparen{\mu_N}_{N \geq 1}$ and $\mu$ be random probability measures on $\RR^k$. Let $\MM_N$ and $\MM$ denote the laws of $\mu_N$ and $\mu$ respectively; i.e.~$\MM_N\insquare{B} = \PP\inbraces{ \omega: \mu_{N,\omega} \in B  }$, for some measurable set $B \subseteq M_1(\RR^k)$. For each $N$, denote by $\LL_N$ the joint law of $\MM_N$ and $\MM$. We say that $\mu_N$ \emph{converges weakly in probability} to $\mu$, and write $\mu_N \overset{\textnormal{in prob.}}{\longsquiggly } \mu$, if for all continuous and bounded functions (or equivalently all bounded Lipschitz functions) $f : \RR^k \rightarrow \RR$, it holds that for all $\epsilon > 0$,
    \begin{align*}
        \LL_N \inbraces{  \omega : \abs{\mu_{N,\omega} f - \mu_{\omega} f } > \epsilon  } \longrightarrow 0, \quad \textnormal{as } N \rightarrow \infty.
    \end{align*}
\end{definition}
In what follows, we find it helpful to suppress the dependence of the random measures on $\omega$, and say that $\mu_N \overset{\textnormal{in prob.}}{\longsquiggly } \mu$, if for all continuous bounded $f$,     
\begin{align*}
    \mu_N\insquare{f} \xrightarrow{\textnormal{in prob.}} \mu\insquare{f}, \quad 
\end{align*} 
where the `probability is with respect to the joint measure of the randomness in $\mu_N$ and $\mu$'.


By choosing $L = M = 1$, $p = 1$ in \eqref{eq:projResult_asymptotic} and using Chebyshev's inequality, we have the following asymptotic statement.

\begin{corollary}
    In the notation of Theorem \ref{thm:ProjResult_DisorderedCase_supOverBL_LM_2p}, suppose \eqref{eq:projResult_TSOC_hypotheses} holds with $c_1, c_2 = K/N$. Then 
    \begin{align}
        \cL(\Theta^\top \vec{x} \, | \, \Theta) \overset{\textnormal{in prob.}}{\longsquiggly } \cL(\sqrt{q}\vec{z} + \sqrt{\rho - q}\bxi \, | \, \vec{z} )
        \label{eq:projResult_asymptotic_ifDirection}
    \end{align}
    where the probability is with respect to the joint product measure of $\cL(\Theta)$ and  $\cL(\vec{z})$. 
\end{corollary}

A partial converse to Theorem \ref{thm:ProjResult_DisorderedCase_supOverBL_LM_2p} is available in the following form.

\begin{proposition}
    Suppose \eqref{eq:projResult_asymptotic_ifDirection} holds for any $k$, then
    \begin{align}
        \frac{\norm{\vec{x}}^2}{N} \xrightarrow{\textnormal{in prob.}} \rho, \quad \textnormal{ and }\quad  \frac{\vec{x}^1 \boldsymbol{\cdot} \vec{x}^2}{N} \xrightarrow{\textnormal{in prob.}} q.
        \label{eq:projResult_converse_inProb}
    \end{align}
    If additionally, for some $r \in (0, \infty)$, it holds that
    \begin{align}
        \sup_{N \geq 1} \inangle{ \abs{ \frac{\norm{\vec{x}}^2}{N} }^r } < + \infty,
        \label{eq:projResult_converse_UI_in_r_hypothesis}
    \end{align}
    then also
    \begin{align}
        \frac{\norm{\vec{x}}^2}{N} \xrightarrow{L_r} \rho, \quad \textnormal{ and }\quad  \frac{\vec{x}^1 \boldsymbol{\cdot} \vec{x}^2}{N} \xrightarrow{L_r} q.
        \label{eq:projResult_converse_inLr}
    \end{align}
    \label{prop:projResult_converse}
\end{proposition}

Notions of thin-shell and overlap concentration are prevalent in the high-temperature phases of spin glass theory. It is then useful to give the form of the main result Theorem \ref{thm:ProjResult_DisorderedCase_supOverBL_LM_2p} in the \emph{disordered} setting. The following two results are were used in a companion paper \cite{wee2022local}; they are not needed for understanding the rest of the paper.

In the disordered setting, the distribution $\inangle{\cdot}$ of $\vec{x}$ is itself a random (Gibbs) measure, and is defined conditionally upon the disorder, with randomness denoted by $\Eb_{\ud}$. In what follows we will always consider the `cavity-field' setting, where the disorder is assumed to be independent of the projection directions $\Theta$. The does not bring much additional difficulty in the proof, which will not be repeated---in the arguments in Section \ref{sec:proof_main_result}, we simply replace all occurrences of $\Eb_{\Theta}$ (resp.~$\Eb_{\Theta} \Eb_{\vec{z}}$) with $\Eb_{\Theta} \Eb_\ud$ (resp.~$\Eb_{\Theta} \Eb_{\vec{z}} \Eb_\ud$).

\begin{corollary}
    Let $0 \leq q < \rho$ be constants such that
\begin{align}
    \Eb_{\textnormal{d}} \inangle{ \inparen{\frac{\norm{\vec{x}}^2}{N} - \rho}^{2}  } \leq c_1, \quad \Eb_{\textnormal{d}} \inangle{ \inparen{\frac{\vec{x}^1 \boldsymbol{\cdot} \vec{x}^2}{N} - q}^{2}  } \leq c_2,
    \label{eq:projResult_TSOC_hypotheses_disordered}
\end{align}
for some numbers $c_1$, $c_2$ that may depend on $N$. Then
\begin{enumerate}
    \item \eqref{eq:projResult_partialAsymptotic} holds with $\Eb_{\Theta}$ replaced by $\Eb_{\Theta} \Eb_\ud$, and
    \item \eqref{eq:projResult_asymptotic} holds with $\Eb_{\Theta} \Eb_{\vec{z}}$ replaced by $\Eb_{\Theta} \Eb_{\vec{z}} \Eb_\ud$.
\end{enumerate}
\label{thm:projResult_disordered}
\end{corollary}

Similarly, we can obtain a disordered analog of the converse. Instead of stating the result in full generality, we restrict for simplicity to the case when the coordinates of $\vec{x}$ have bounded support. 

\begin{corollary}
    Let $\vec{x}$ be a random vector drawn from a disordered distribution $\inangle{\cdot}$. Let $\vec{x}$ have coordinates in some compact subset of $\RR$ a.e.~$\Eb_{\ud}\inangle{\cdot}$. Suppose that \eqref{eq:projResult_asymptotic_ifDirection} holds, where the probability is with respect to the joint product measure of the disorder, $\cL(\Theta)$, and $\cL(\vec{z})$. Then for every $r \geq 1$, 
    \begin{align}
        \Eb_{\textnormal{d}} \inangle{ \inparen{\frac{\norm{\vec{x}}^2}{N} - \rho}^{r}  } \longrightarrow 0, \quad \textnormal{and}\quad \Eb_{\textnormal{d}} \inangle{ \inparen{\frac{\vec{x}^1 \boldsymbol{\cdot} \vec{x}^2}{N} - q}^{r}  } \longrightarrow 0.
    \end{align}
\end{corollary}

\subsection{Outline of proofs}
\label{sec:outline}

First consider the one-dimensional case $k= 1$. Here the projection direction is the random vector $\btheta \equiv \Theta \sim \cN(0, N^{-1}I)$, independent of the random vector $\vec{x} \in \RR^N$. We assume the thin-shell and overlap concentration hypotheses 
\begin{align*}
    \frac{1}{N} \norm{\vec{x}}^2 \simeq \rho; \quad\quad \frac{1}{N} \vec{x}^1\boldsymbol{\cdot} \vec{x}^2 \simeq q.
\end{align*}
In this section `$\simeq$' is imprecise and read as `is close to'. We discuss the asymptotic case \eqref{eq:projResult_asymptotic} first, where the goal is to show $\cL(\btheta^\top \vec{x} \,  | \, \btheta) \simeq \cL(\sqrt{q} z + \sqrt{\rho - q}\xi \, | \, z)$, where $z, \xi$ are independent standard Gaussians, independent of everything else. Consider the \emph{unconditional} joint distribution $\cL(\btheta^\top \vec{x}^1, \btheta^\top \vec{x}^2)$, where $\vec{x}^1, \vec{x}^2$ are independent copies of $\vec{x}$. By conditioning on $\vec{x}^1, \vec{x}^2$, the hypotheses yield
\begin{align}
    \cL\inparen{ \begin{bmatrix} 
    \btheta^\top \vec{x}^1 \\
    \btheta^\top \vec{x}^2
    \end{bmatrix} } &= \Eb_{\vec{x}^1, \vec{x}^2} \, \cN\inparen{ \begin{bmatrix} 
    0 \\
    0
    \end{bmatrix}, \, \frac{1}{N}\begin{bmatrix}
    \norm{\vec{x}^1}^2 & \vec{x}^1 \boldsymbol{\cdot} \vec{x}^2  \\
    \vec{x}^1 \boldsymbol{\cdot} \vec{x}^2 & \norm{\vec{x}^2}^2
    \end{bmatrix}} \simeq \cN\inparen{ \begin{bmatrix} 
    0 \\
    0 
    \end{bmatrix}, \, \begin{bmatrix}
    \rho & q \\
    q & \rho
    \end{bmatrix}} = \cL\inparen{ \begin{bmatrix} 
        \sqrt{q}z + \sqrt{\rho - q}\xi^1 \\
        \sqrt{q}z + \sqrt{\rho - q}\xi^2
        \end{bmatrix} }.
    \label{eq:intro_projResult_unconditional}
\end{align}
We now pass from the unconditional joint distribution to the conditional distribution. To simplify notation, use the shorthand $Y_N^\ell := \btheta^\top \vec{x}^\ell $ and $Y^\ell := \sqrt{q}z + \sqrt{\rho - q}\xi^\ell$. Notice here only $\vec{x}$ and $\xi$ are replicated. Note also that the replicas are used in the following way: $\inparen{\Eb_\xi Y}^2 = \Eb_\xi \insquare{Y^1 Y^2}$, where in the RHS the expectation is over the product measure $\cL(\xi)^{\otimes 2}$.

Consider test functions $g : \RR \rightarrow \RR$ and write, using the elementary identity $(a-b)^2 = a^2 - b^2 - 2(ab - b^2)$,
\begin{align}
    \Eb_z \Eb_{\btheta} \insquare{ \insquare{\inangle{g(Y_N)} - \Eb_\xi g(Y)  }^2 } &\leq \abs{ \Eb_z \Eb_{\btheta} \inangle{g(Y_N^1)g(Y_N^2)}  - \Eb_z \Eb_{\btheta} \Eb_\xi \insquare{ g\inparen{ Y^1 } g\inparen{ Y^2  } } } \nonumber \\
    &\quad\quad\quad + 2\abs{\Eb_z \Eb_{\btheta} \inangle{g(Y_N)}\Eb_\xi g(Y) - \Eb_z \Eb_{\btheta} \Eb_\xi\insquare{ g\inparen{ Y^1 } g\inparen{ Y^2  } } }.
    \label{eq:intro_projResult_applyingHoeffding}
\end{align}
The first term on the RHS is small exactly because of \eqref{eq:intro_projResult_unconditional}. The second term on the RHS is small due also to \eqref{eq:intro_projResult_unconditional}, along with the fact that when the joint distribution converges, the marginal distributions also converge. (The latter is made precise by Lemma \ref{lemma:jointConverge_implies_twistedMarginalConverge(Wasserstein)}). Thus, from \eqref{eq:intro_projResult_applyingHoeffding} we get
\begin{align*}
    \Eb_z \Eb_{\btheta} \insquare{ \insquare{\inangle{g(\btheta^\top\vec{x})} - \Eb_\xi g(\sqrt{q}z + \sqrt{\rho - q}\xi)  }^2 } \longrightarrow 0.
\end{align*}
We use a multivariate normal approximation by an infinitesimal exchangeable pair version of Stein's method \cite{meckes2009stein} \cite{reinert2009multivariate} to control \eqref{eq:intro_projResult_unconditional}, which leads to a non-asymptotic result of the previous display. This finishes the outline for \eqref{eq:projResult_asymptotic} when $k = 1$, $p = 1$.


To attain \eqref{eq:projResult_partialAsymptotic}, we identify the common randomness $\sqrt{q}z$ in the RHS of \eqref{eq:intro_projResult_unconditional} with the limiting law of $\btheta^\top \inangle{\vec{x}}$. Indeed, under the overlap concentration hypothesis,
\begin{align}
    \cL\inparen{ \btheta^\top \inangle{\vec{x}}   } &= \cN\inparen{ 0, \frac{1}{N}\norm{\inangle{\vec{x}}}^2 } \simeq \cN(0, q).
    \label{eq:intro_projResult_identify_z_with_projectionOfMean}
\end{align}
The triangle inequality and \eqref{eq:intro_projResult_unconditional} then imply
\begin{align*}
    \cL\inparen{ \begin{bmatrix} 
        \btheta^\top \vec{x}^1 \\
        \btheta^\top \vec{x}^2
        \end{bmatrix} } \simeq \cL\inparen{ \begin{bmatrix} 
            \btheta^\top \inangle{\vec{x}} + \sqrt{\rho - q}\xi^1 \\
            \btheta^\top \inangle{\vec{x}} + \sqrt{\rho - q}\xi^2
        \end{bmatrix} }.
\end{align*}
Repeating the arguments in \eqref{eq:intro_projResult_applyingHoeffding}, where it is now unnecessary to condition on $z$, leads to the desired statement \eqref{eq:projResult_partialAsymptotic}.

The situation for the multi-dimensional projections is a straightforward generalization of the above arguments. In the case when $k = 2$, we are interested in the large-$N$ distribution of $ \Theta^\top\vec{x} := (\btheta_1^\top \vec{x}, \btheta_2^\top \vec{x})$, where $\btheta_i$ are independent copies of $\btheta$. The replicated system to consider is
\begin{align*}
    \cL\inparen{ \begin{bmatrix} 
    \btheta_1^\top \vec{x}^1 \\
    \btheta_2^\top \vec{x}^1 \\
    \btheta_1^\top \vec{x}^2 \\
    \btheta_2^\top \vec{x}^2
    \end{bmatrix} } \simeq \cN\inparen{ \begin{bmatrix} 
    0 \\
    0 \\
    0 \\
    0
    \end{bmatrix}, \, \begin{bmatrix}
    \rho & 0 & q  & 0 \\
    0 & \rho & 0 & q \\
    q  & 0  & \rho & 0 \\
    0 & q & 0 & \rho
    \end{bmatrix}} \equiv 
    \cL\inparen{ \begin{bmatrix} 
    \sqrt{q}z_1 + \sqrt{\rho - q}\xi_1^{1} \\
    \sqrt{q}z_2 + \sqrt{\rho - q}\xi_2^{1} \\
    \sqrt{q}z_1 + \sqrt{\rho - q}\xi_1^{2} \\
    \sqrt{q}z_2 + \sqrt{\rho - q}\xi_2^{2}
    \end{bmatrix} },
\end{align*}
where the approximations follow similarly by the thin-shell and overlap concentration hypotheses, and where $z_j$'s and $\xi_j^\ell$'s are independent standard Gaussians, independent of everything else. Notice the convention that the subscripts $j$ refer to randomness derived from projection direction $\btheta_j$, whereas the superscript $\ell$ is reserved for replicas. For this reason, we never write $z^1$ or $z^2$ because it will become apparent (as in \eqref{eq:intro_projResult_identify_z_with_projectionOfMean}) that the randomness in $z_j$'s comes solely from the projection direction (and possibly disorder), which is never replicated.

Repeating the arguments in \eqref{eq:intro_projResult_applyingHoeffding} analogously, where now the outer expectation is over $\vec{z} = (z_1, z_2)^\top$ and $\Theta = (\btheta_1, \btheta_2)$, yields
\begin{align*}
    \Eb_\vec{z} \Eb_{\Theta} \insquare{ \insquare{\inangle{g(\Theta^\top\vec{x})} - \Eb_\xi g(\sqrt{q}\vec{z} + \sqrt{\rho - q}\bxi)  }^2 } \longrightarrow 0,
\end{align*}
where $\bxi = (\xi_1, \xi_2)^\top$. This is indeed the multi-dimensional projection result \eqref{eq:projResult_asymptotic}. The corresponding partially asymptotic statement \eqref{eq:projResult_partialAsymptotic} follows by identifying the limiting distribution of $\inparen{ \btheta_1^\top \inangle{\vec{x}}, \btheta_2^\top \inangle{\vec{x}} }$ with $\inparen{\sqrt{q}z_1, \sqrt{q}z_2}$, as was done in \eqref{eq:intro_projResult_identify_z_with_projectionOfMean}.

So far we have assumed $p = 1$. Fortunately, for higher moments $p \geq 1$, the complication is only algebraic. We simply have to replicate the system $2p$ times instead. This follows because the elementary identity used in \eqref{eq:intro_projResult_applyingHoeffding} generalizes for even powers. For instance when $p = 2$, we have $(a-b)^4 = \inparen{ a^4 - b^4} - 4\inparen{ a^3 b - b^4  } + 6 \inparen{ a^2 b^2 - b^4  } - 4\inparen{ ab^3 - b^4  }$, which informs us that we have to replicate the system at most four times.

\subsection{Further background}
\label{sec:further_background}

In this section we give further remarks on how our results fit in the context of existing projection results; on the relation between overlap concentration and other geometric assumptions on the distribution of $\vec{x}$; on the difficulty in accommodating a non-zero mean; and on the advantage of our proof technique over approaches that are more reliant on concentration of measure.



\parag{Relation to other results.} 
\begin{itemize}
    \item In the spin glass literature, where such projection results are referred to as 'central limit theorems for cavity-fields', the earliest rigorous proofs for the SK model and its $p$-spin variants can be traced to Talagrand \cite[Theorem 1.7.11]{talagrand2010mean}, and then to Chen \cite[Theorem 2]{chen2010central}, Chatterjee \cite[Theorem 1.5]{chatterjee2010spin}, and \cite[Theorem 5.1]{chen2022tap}. Such results are crucial for the cavity method in high-temperature or \emph{replica-symmetry}, which is believed to be synonymous with overlap concentration.

    In \cite{talagrand2010mean} and \cite{chen2010central}, the test functions are the the infinitely differentiable functions, and the rate is stronger than the disordered version Corollary \ref{thm:projResult_disordered}, with the RHS scaling like $O(N^{-p})$. However, such results are specific to the SK Gibbs measure, whereas Corollary \ref{thm:projResult_disordered} holds more generally. Moreover, previous results are for the one-dimensional $k = 1$ case, whereas Corollary \ref{thm:projResult_disordered} provisions for $k \geq 1$, which also illustrates the conditional (on the disorder) independence of the joint cavity-fields in the limit, a fact which has perhaps not been previously exposed in the literature.  
    
     
    \item In another line of work involving projections of high-dimensional distributions onto random lower-dimensional subspaces, our results are closely related to those of \cite{diaconis1984asymptotics}, \cite{reeves2017conditional}, and \cite{dumbgen2013low} who considered the zero overlap setting. Taking $q = 0$ in Theorem \ref{thm:ProjResult_DisorderedCase_supOverBL_LM_2p} essentially recovers these results, up to different probability distance metrics and rates. In particular, the sufficiency and necessity of thin-shell and overlap concentration (on zero) was obtained in \cite{dumbgen2013low}, and this is subsumed by Theorem \ref{thm:ProjResult_DisorderedCase_supOverBL_LM_2p} and Propostion \ref{prop:projResult_converse}.

    Projection results are also available under more restrictive settings than overlap concentration (see next point for clarification) \cite{sudakov1978typical}, \cite{von1997sudakov}, \cite{bobkov2003concentration}, \cite{bobkov2010theorem}, \cite{meckes2009quantitative_1}, \cite{meckes2012approximation_2}, \cite{meckes2012projections_3}. An elegant result from Meckes \cite{meckes2012projections_3} states that under thin-shell and bounded covariance eigenvalues on $\vec{x}$, there is a threshold $k = 2 \log N/(\log\log N)$ under which Gaussian projections persist, and this is sharp for the bounded-Lipschitz metric.

    Strictly speaking, the projection matrix $\Theta$ used in several of the aforementioned papers is drawn from Haar measure on the Stiefel manifold, which is not identical to the Gaussian setting considered here and in \cite{reeves2017conditional}, \cite{dasgupta2012concentration}, \cite{dumbgen2013low}. However, by the closeness of the subsets of columns of matrices drawn uniformly from the orthogonal group and multivariate Gaussians, the two settings are not far off (see for instance \cite{meckes2019book} Chapter 2.3, or \cite{eaton1989group} Chapters 7, 8). Note also that the results in this paper are likely extendable to the Stiefel manifold setting, since it is mainly the rotational-invariance of the Gaussian that is used in the construction of the exchangeable pair in the Stein's method proof of Lemma \ref{lemma:PPN_to_QQ}.

\end{itemize}

\parag{Overlap concentration.} The overlap concentration assumption provides a more general setting than previously considered settings of isotropy and bounded covariance eigenvalues. Let $\vec{x}$ be a random vector in $\RR^N$. Assume that $\sup_N q_N := \inangle{N^{-1} \vec{x}^1 \boldsymbol{\cdot} \vec{x}^2} < \infty$. Denote by $\blambda = (\lambda_i)_{i \leq N}$ the eigenvalues of $\Cov \vec{x} = \inangle{ \vec{x}\vec{x}^\top } - \inangle{\vec{x}}\inangle{\vec{x}}^\top$. We have by expanding,
\begin{align}
    \inangle{\inparen{ \frac{1}{N}\vec{x}^1 \boldsymbol{\cdot}\vec{x}^2 - q_N }^2 } &= \frac{1}{N^2} \inangle{ \inparen{ (\vec{x}^1 - \inangle{\vec{x}})^\top (\vec{x}^2 - \inangle{\vec{x}})  }^2  } + \frac{2}{N^2} \inangle{\vec{x}}^\top \Cov \vec{x} \inangle{\vec{x}}
    \nonumber \\
    &\leq \frac{1}{N^2} \norm{\blambda}^2 + \frac{2}{N} \inparen{\max_{i \leq N} \lambda_i} q_N,
    \label{eq:overlapConc_CovEigval_bound}
\end{align}
whence it is seen that a sufficient condition for overlap concentration is for $N^{-1}\norm{\blambda} \rightarrow 0$. 
\begin{enumerate}
    \item If $\vec{x}$ is in the isotropic position, i.e.~$\Cov \vec{x} = I$, then $\norm{\blambda} = \sqrt{N}$ and RHS in \eqref{eq:overlapConc_CovEigval_bound} goes to zero as $N \rightarrow \infty$, implying overlap concentration.
    \item More generally, by \eqref{eq:overlapConc_CovEigval_bound}, overlap concentration holds in the bounded covariance eigenvalues setting, i.e.~when the $\lambda_i$'s are required to be uniformly bounded, independently of $N$. This is the setting of \cite{meckes2012projections_3}, \cite{bobkov2010theorem} Theorems 1 and 2 in the zero mean case, and in the examples considered in \cite{dasgupta2012concentration}. Related observations are made in \cite[Section 1-B]{reeves2017conditional}.
\end{enumerate}
On the other hand, it is easy to construct distributions which are neither isotropic, nor have bounded covariance eigenvalues, but for which overlap concentrates. For instance: any random vector with $\sup_N q_N < \infty$ and with covariance $\diag\inparen{\sqrt{N}\rho_1, \rho_2, \dots, \rho_N}$, where $(\rho_i)_{i \geq 1}$ is any sequence of nonnegative real numbers with $\sup_i \rho_i < \infty$, has overlap concentrating in $L_2$ by \eqref{eq:overlapConc_CovEigval_bound}.

\parag{Centering.} If $\vec{x}$ satisfies overlap concentration on $q$ and thin-shell on $\rho$, then its centered version $\vec{w} = \vec{x} - \inangle{\vec{x}}$ will have overlap concentration on $0$ and thin-shell on $\rho -q$. The existing zero overlap results then imply that $\cL\inparen{ \Theta^\top \vec{w} \, | \, \Theta } \simeq \cN(0, \rho - q)$. However, it is in general a non-trivial task to recover the uncentered projection result from the centered version. To wit, weak convergence is not generally closed under addition, unless additional information is provided. For instance: that one of the sequences is converging to a constant, which allows the use of the `converging together'/Slutsky's lemma; or if the converging sequences are mutually independent. Neither of these are applicable here.

In \cite[Theorem 1.7.11]{talagrand2010mean} a centered version of the projection result is proved, and it was left as \cite[Research Problem 1.7.12]{talagrand2010mean} to ascertain if it was even true that the general (uncentered) case holds. This was subsequently answered in the affirmative by Chen \cite{chen2010central} Theorem 2, and also Chatterjee \cite{chatterjee2010spin} Theorem 1.5. We remark that the strategy used by Chen involving `mean-translated' test functions may provide an alternative route to uncenter such weak convergence statements.

\parag{Comparison to existing proof techniques.} 
Instead of replicating the system as encouraged by the Hoeffding technique \eqref{eq:opening_Hoeffding}, there exist other approaches that show the quenched projection statements by an `annealed $+$ concentration'' approach in the following sense (see e.g.~proof in \cite{meckes2012projections_3}, \cite{bobkov2003concentration}, or \cite{dasgupta2012concentration}). For $\QQ$ some target distribution, write by triangle inequality
\begin{align*}
    \textnormal{dist}\!\inparen{ \cL(\Theta^\top\vec{x} \, | \, \Theta), \, \QQ  } &\leq \textnormal{dist}\!\inparen{ \cL(\Theta^\top \vec{x}), \, \QQ } + \textnormal{dist}( \cL(\Theta^\top \vec{x}), \, \cL(\Theta^\top \vec{x} \, | \, \Theta) ).
\end{align*}
The first term is small by annealed results such as Lemma \ref{lemma:PPN_to_QQ}. Here, because the system is not replicated, the thin-shell condition typically suffices. However, the second term is the bottleneck, and is typically controlled by ad-hoc methods such as Gaussian concentration of Lipschitz functions. The Lipschitz constants involved are, in turn, often related to the maximum covariance eigenvalues. Consequently, this approach leads to the restrictive conditions imposed on the spectrum of $\Cov \vec{x}$.

\subsection{Notation}

For a vector $\vec{x}$ on $\RR^N$ we write $\norm{\vec{x}}$ for the Euclidean norm. For a random vector $\vec{x}$, the law of $\vec{x}$ is written $\cL(\vec{x})$. The conditional distribution of $\vec{x}$ given $\vec{y}$ is written $\cL(\vec{x} \, | \, \vec{y})$. Indicator functions of a set $A$ are denoted by $\boldsymbol{1}_A(\cdot)$. We use the shorthand $R^{m}_{\rho, q}$ for the $m \times m$ ``replica-symmetric'' matrices:
\begin{align}
    R^{m}_{\rho, q} &:= \begin{bmatrix}
    \rho  & q  & \dots & q  \\
    q  & \rho  & \dots & q  \\
    \vdots & \vdots  & \ddots & \vdots \\
    q  & q  & \dots & \rho 
    \end{bmatrix}.
    \label{eq:RS_matrix_shorthand}
\end{align}
Matrix direct product and sum are written $A \otimes B$ and $A \oplus B$ respectively. On the space of real square matrices we consider the Hilbert-Schmidt or Frobenius inner product $\inangle{A, B}_{\textnormal{HS}} = \Tr (AB^\top)$ which induces the norm $\norm{A}_{\textnormal{HS}} = \sqrt{\Tr(AA^\top)}$.


\begin{definition}[Lipschitz norms]
For $g : \RR^N \rightarrow \RR$, 
\begin{align*}
    \norm{g}_{\textnormal{Lip}} = \sup_{\bx \neq \by} \frac{ \abs{g(\bx) - g(\by)} }{\abs{\bx - \by}}.
\end{align*}
\end{definition}

\begin{definition}[Wasserstein distances]
The Wasserstein-$p$ distance between two probability measure $\PP$ and $\QQ$ on a metric space with metric $d$ is given by
\begin{align*}
    \textnormal{W}_p(\PP, \QQ) &= \inparen{ \inf_{\mu \in \Pi(\PP, \QQ)} \inbraces{ \int d(x,y)^{p} \; \ud \mu} }^{1/p},
\end{align*}
where $\Pi(\PP, \QQ) $ denotes the set of all couplings of $\PP$ and $\QQ$, with integrable distance function. We will mostly use the case $p = 1$, where we denote for short $\textnormal{W} \equiv \textnormal{W}_1$. Kantorovich-Rubinstein duality gives the equivalence
\begin{align}
    \textnormal{W}( \cL\!\inparen{ \vec{x} },  \cL\!\inparen{ \vec{y} }) &= \sup_{\norm{g}_{\textnormal{Lip}} \leq 1} \abs{\Eb g(\vec{x}) - \Eb g(\vec{y})}.
    \label{eq:KR_duality}
\end{align}
\end{definition}


\section{Proof of Theorem \ref{thm:ProjResult_DisorderedCase_supOverBL_LM_2p}}
\label{sec:proof_main_result}

As outlined in Section \ref{sec:outline}, in order to prove statements for conditional weak convergence, we replicate the unconditional distributions appropriately. For integers $k, p \geq 1$, define the following probability measures on $\RR^{2kp}$:
\begin{align}
    \PP_{N} &:= \cL\inparen{ \Theta^\top \vec{x}^{1}, \dots, \Theta^\top \vec{x}^{2p} }; \nonumber\\
    \QQ_{N} &:= \cN\inparen{ \Theta^\top\inangle{\vec{x}} \boldsymbol{1}, \inparen{ (\rho - q) I_k }^{\oplus 2p}   } \equiv \cL\inparen{  \Theta^\top\inangle{\vec{x}} + \sqrt{\rho - q}\bxi^1,\dots, \Theta^\top\inangle{\vec{x}} + \sqrt{\rho - q}\bxi^{2p}    } ; \nonumber\\
    \QQ &:= \cN\inparen{\boldsymbol{0}, R^{2p}_{\rho, q} \otimes I_k} 
    \label{eq:randomProjections_definition_PPN_QQN_QQ},
\end{align}
where $\inparen{ \bxi^\ell }_{\ell \leq 2p}$ are independent standard Gaussian random vectors in $\RR^k$, independent of everything else, and $R^{2p}_{\rho,q}$ is defined in \eqref{eq:RS_matrix_shorthand}. Note that none of these are random probability measures.

The proof has three main parts:
\begin{enumerate}
    \item First bound the Wasserstein distance between $\PP_N$ and $\QQ$ (Lemma \ref{lemma:PPN_to_QQ}). The main technical tool is an infinitesimal version of Stein's method of exchangeable pairs for normal approximation \cite{meckes2009stein} \cite{reinert2009multivariate}. The relevant results are described in Section \ref{sec:multivar_normal_SteinsMethod_infExchangeablePairs}.
    
    \item Next, bound the Wasserstein distance between $\QQ_N$ and $\QQ$ (Lemma \ref{lemma:QQN_to_QQ}). This is a comparison between two Gaussian distributions. We mainly need to recognize that under the overlap concentration hypothesis, $\Theta^\top \inangle{\vec{x}}$ is close to $\sqrt{q}\vec{z}$ for large $N$, where $\vec{z}$ is a standard Gaussian vector in $\RR^{k}$, and that $\cL\inparen{ \sqrt{q}\vec{z} + \sqrt{\rho - q} \bxi^1, \dots, \sqrt{q}\vec{z} + \sqrt{\rho - q}\bxi^{2p}  }$ has covariance matrix $R_{\rho,q}^{2p} \otimes I_k$.

    By triangle inequality, we can then bound the Wasserstein distance between $\PP_N$ and $\QQ_N$, giving \ref{thm:UnconditionalProjResult_X_PPNtoQQN}.

    \item Finally, show using the argument in \eqref{eq:intro_projResult_applyingHoeffding} that $\inangle{g\inparen{\Theta^\top \vec{x}}} \simeq \Eb_{\xi} \insquare{ g\inparen{\Theta^\top\inangle{X} + \sqrt{\rho - q}\xi}  }$ in $L_{2p}$ wrt.~the randomness in $\Theta$, and that $\inangle{g\inparen{\Theta^\top \vec{x}}} \simeq \Eb_{\xi} \insquare{ g\inparen{\sqrt{q}\vec{z} + \sqrt{\rho - q}\xi}  }$ in $L_{2p}$ wrt.~the randomness in $\Theta$ and $\vec{z}$. This yields Theorem \ref{thm:ProjResult_DisorderedCase_supOverBL_LM_2p}.
\end{enumerate}


\subsection{Unconditional joint distribution convergence in Wasserstein distance}
\label{sec:finite_kdim_gaussian_uncentered}

The purpose of this section is establish Lemmas \ref{lemma:PPN_to_QQ} and \ref{thm:UnconditionalProjResult_X_PPNtoQQN} that show that in Wasserstein distance, $\PP_N$ is close $\QQ$, and $\PP_N$ is close to $\QQ_N$.

\begin{lemma}[$\PP_N$ to $\QQ$]
Suppose that \eqref{eq:projResult_TSOC_hypotheses} is satisfied. Let $\PP_{N}$ and $\QQ_N$ be the probability measures on $\RR^{2kp}$ given by \ref{eq:randomProjections_definition_PPN_QQN_QQ}. Then for every integer $p \geq 1$, 
\begin{align}
    \sup_{\norm{g}_{\textnormal{Lip}} \leq 1} \abs{ \PP_{N}\insquare{ g } - \QQ \insquare{g} } &\leq \frac{32 p^2 k  /\sqrt{\rho - q}}{N - 1} \inparen{ d_1(c_1) + d_2(c_2)  },
    \label{eq:ProjectionResult_PPN_to_QQ_RATE}
\end{align}
where $d_1$, $d_2$ are defined as in \eqref{eq:d1_d2_rates}.
\label{lemma:PPN_to_QQ}
\end{lemma}

The proof of Lemma \ref{lemma:PPN_to_QQ} uses Stein's method for multivariate normal approximation, in particular an infinitesimal exchangeable pairs approach \cite{meckes2009stein}, \cite{reinert2009multivariate}---the relevant results and the proof are given in Appendix \ref{sec:suppProofs_for_randomProjs}.

\begin{lemma}[$\QQ_N$ to $\QQ$]
Suppose that \eqref{eq:projResult_TSOC_hypotheses} is satisfied. Let $\QQ_{N}$ and $\QQ$ be the probability measures on $\RR^{2kp}$ given by \eqref{eq:randomProjections_definition_PPN_QQN_QQ}. Then for every integer $p \geq 1$, 
\begin{align*}
    \sup_{\norm{g}_{\textnormal{Lip}} \leq 1} \abs{ \QQ_{N} \insquare{g} - \QQ\insquare{g} } &\leq 2pk\inparen{ \frac{1}{\sqrt{q}}\sqrt{c_2} \boldsymbol{1}_{q > 0} + c_{2}^{1/4} \boldsymbol{1}_{q = 0}  }.
\end{align*}
\label{lemma:QQN_to_QQ}
\end{lemma}

\begin{proof}
Let $\vec{z}, \vec{\xi}^{\ell}$, $\ell \leq 2p$ be standard Gaussian vectors in $\RR^k$, independent of each other. Let $R_{1,2} = N^{-1}\vec{x}^1 \boldsymbol{\cdot} \vec{x}^2$. Note that $\inangle{R_{1,2}} = \norm{\inangle{\vec{x}}}^{2}_{2} \geq 0$. We may write
\begin{align*}
    \QQ_N &= \cL\inparen{  \inparen{ \sqrt{\inangle{R_{1,2}}} \vec{z} + \sqrt{\rho -q} \vec{\bxi}^{\ell} }_{\ell \leq 2p}  }\\
    \QQ &= \cL\inparen{  \inparen{ \sqrt{q} \vec{z} + \sqrt{\rho -q} \vec{\xi}^{\ell} }_{\ell \leq 2p}  }.
\end{align*}
Let $\tilde{\mu} \in \Pi(\QQ_N, \QQ)$ denote the coupling represented by
\begin{align*}
    \tilde{\mu} = \cL\inparen{ \inparen{ \sqrt{\inangle{R_{1,2}}} \vec{z} + \sqrt{\rho -q} \vec{\xi}^{\ell} }_{\ell \leq 2p}  , \,  \inparen{ \sqrt{q} \vec{z} + \sqrt{\rho -q} \vec{\xi}^{\ell} }_{\ell \leq 2p} }.
\end{align*}
By Kantorovich-Rubinstein duality \eqref{eq:KR_duality}, we have
\begin{align*}
    \textnormal{W}\inparen{\QQ_N, \QQ} &= \inf_{\mu \in \Pi(\QQ_N, \QQ)} \int \norm{\vec{w} - \vec{v}} \, \mu \inparen{\ud \vec{w}, \ud \vec{v}} \\
    &\leq \int \norm{\vec{w} - \vec{v}}_{1} \, \tilde{\mu} \inparen{\ud \vec{w}, \ud \vec{v}} \\
    &= \Eb_{\vec{z}} \norm{ \inparen{ \inparen{ \sqrt{\inangle{R_{1,2}}} - \sqrt{q} }\vec{z} }_{\ell \leq 2p}  }_{1} \\
    &= 2pk \sqrt{\frac{2}{\pi}} \abs{   \sqrt{\inangle{R_{1,2}}} - \sqrt{q}}.
\end{align*}
When $q = 0$, we have $\sqrt{\inangle{R_{1,2}}} \leq \inangle{\inparen{ R_{1,2} - 0}^{2}}^{1/4} \leq c_2^{1/4}$. When $q > 0$,
\begin{align*}
    \abs{\sqrt{\inangle{R_{1,2}}} - \sqrt{q}} = \frac{\abs{\inangle{R_{1,2}} - q}}{\sqrt{\inangle{R_{1,2}}} + \sqrt{q}} \leq \frac{1}{\sqrt{q}} \sqrt{ c_2 }.
\end{align*}
This completes the proof.
\end{proof}

\begin{lemma}[$\PP_N$ to $\QQ_N$]
Let $\PP_{N}$ and $\QQ_{N} $ be the probability measures on $\RR^{2kp}$ given by \eqref{eq:randomProjections_definition_PPN_QQN_QQ}. Then for every integer $p \geq 1$,
\begin{align}
    \sup_{\norm{g}_{\textnormal{Lip}} \leq 1} \abs{ \PP_N \insquare{g} - \QQ_N\insquare{g} } &\leq \frac{34 p^2 k  /\sqrt{\rho - q}}{N - 1} \inparen{ d_1(c_1) +   \frac{1 + \sqrt{q}}{\sqrt{q}} d_2\boldsymbol{1}_{q > 0} + N\inparen{c_2}^{1/4}\boldsymbol{1}_{q = 0}  },
    \label{eq:randomProjections_UncondProj_PPN_to_QQN_RATE}
\end{align}
where $d_1$, $d_2$ are defined as in \eqref{eq:d1_d2_rates}.
\label{thm:UnconditionalProjResult_X_PPNtoQQN}
\end{lemma}

\begin{proof}
The result follows from the triangle inequality $\textnormal{W}(\PP_N, \QQ_N) \leq \textnormal{W}(\PP_N, \QQ) + \textnormal{W}(\QQ_N, \QQ)$ and Lemmas \ref{lemma:PPN_to_QQ} and \ref{lemma:QQN_to_QQ}.
\end{proof}

\subsection{Proof of Theorem \ref{thm:ProjResult_DisorderedCase_supOverBL_LM_2p}}
\label{sec:finite_kdim_gaussian_2/2}

We first state some auxiliary results. The proofs are in Appendix \ref{sec:suppProofs_for_randomProjs}. The following lemma shows that products of bounded Lipschitz functions are Lipschitz.

\begin{lemma}
Let $g : \RR^{k} \rightarrow \RR$ satisfy $\norm{g}_{\textnormal{Lip}} \leq L < \infty$ and $\norm{g}_{\infty} \leq M < \infty$. Then for every integer $r \geq 1$, the function $F_r : \RR^{kr} \rightarrow \RR$, defined by $F_r(x_1,\dots,x_r) = g(x_1)\cdots g(x_r)$, satisfies $\norm{F_r}_{\textnormal{Lip}} \leq r L M^{r-1} $.
\label{lemma:||F_r||_Lip}
\end{lemma}

The next lemma is essentially the statement that when the law of random elements $(A_1, A_2)$ is close to that of $(B_1,B_2)$, then the law of $(A_1, B_2)$ will be close to that of $(A_2, B_2)$.

\begin{lemma}
    Let $Y, Z$ (resp.~$U$) be random elements taking values in a Polish space $S$ (resp.~$S'$). Let $f : (S \times S') \rightarrow \RR^{d}$ be a measurable map. Let $\inparen{ Y^\ell }_{\ell \leq D}$ and $\inparen{ Z^\ell }_{\ell \leq D}$ be independent copies of $Y$ and $Z$ respectively. Let $\PP$, $\QQ$, $\TT_r$ be probability measures on $\RR^{dD}$ defined by $\PP := \cL\inparen{  f(Y^1,U), \dots, f(Y^D, U)  }$, $\QQ := \cL\inparen{  f(Z^1,U), \dots, f(Z^D, U) }$, and $\TT_r := \cL\inparen{  f(Y^1,U), \dots, f(Y^r, U), f(Z^{r+1}, U), \dots, f(Z^{D}, U) }$ for any $r \leq D$. Then
    \begin{align}
        \textnormal{W}(\TT_r, \QQ) \leq \textnormal{W}(\PP, \QQ).
    \end{align}
\label{lemma:jointConverge_implies_twistedMarginalConverge(Wasserstein)}
\end{lemma}

\begin{proofof}{Theorem \ref{thm:ProjResult_DisorderedCase_supOverBL_LM_2p}}
We start with \eqref{eq:projResult_partialAsymptotic}. Fix any $g : \RR^k \rightarrow \RR$ such that $\norm{g}_{\textnormal{Lip}} \leq L$ and $\norm{g}_{\infty} \leq M$. Denote $\vec{g}^{\ell} := g(\Theta^\top \vec{x}^{\ell})$ and $\vec{\bar{g}}^{\ell} := g(\Theta^\top \inangle{\vec{x}} + \sqrt{\rho - q}\bxi^\ell)$ for $1 \leq \ell \leq 2p$, and set $g^{0}, \bar{g}^{0} \equiv 1$. 
Expanding, and using replicas, we have
\begin{align}
    &\Eb_{\Theta} \insquare{ \inparen{ \inangle{ g(\Theta^\top \vec{x})  } - \Eb_{\xi}\insquare{ g\inparen{  \Theta^\top \inangle{\vec{x}} + \sqrt{\rho - q}\bxi }  }  }^{2p}  } \nonumber\\
    &\quad\quad = \sum_{0 \leq r \leq 2p} (-1)^{2p - r} \binom{2p}{r} \Eb_{\Theta} \insquare{ \inangle{\vec{g}^1 \cdots \vec{g}^{r}} \Eb_\xi \insquare{ \vec{\bar{g}}^1 \cdots \vec{\bar{g}}^{2p - r}  } } \nonumber\\
    &\quad\quad = \sum_{1 \leq r \leq 2p} (-1)^{2p - r} \binom{2p}{r} \inparen{  \Eb_{\Theta} \insquare{ \inangle{\vec{g}^1 \cdots \vec{g}^{r}}  \Eb_\xi \insquare{ \vec{\bar{g}}^{r+1} \cdots \vec{\bar{g}}^{2p}  } } - \Eb_{\Theta} \Eb_\xi \insquare{ \vec{\bar{g}}^1 \cdots \vec{\bar{g}}^{2p}  } },
    \label{eq:ProjResult_supOverBL_LM_2p_expansion}
\end{align}
where the last equality follows from the elementary identity: for $n$ even,
\begin{align}
    \binom{n}{0} = -\binom{n}{n} + \binom{n}{n-1} - \binom{n}{n-2} + \cdots - \binom{n}{2} + \binom{n}{1},
    \label{eq:(1-1)^n_expand_identity}
\end{align}
which follows the rearranging the binomial expansion of $(1-1)^n$. For each $1 \leq r \leq 2p$, define 
\begin{align*}
    \TT_{N,r} := \cL\inparen{ \Theta^\top \vec{x}^1, \dots, \Theta^\top \vec{x}^{r}, \Theta^\top\inangle{\vec{x}} + \sqrt{\rho - q}\bxi^{r + 1}, \dots, \Theta^\top\inangle{\vec{x}} + \sqrt{\rho - q}\bxi^{2p}  }.
\end{align*}
Then 
\begin{align*}
    &\abs{ \Eb_{\Theta} \insquare{ \inangle{\vec{g}^1 \cdots \vec{g}^{r}}  \Eb_\xi \insquare{ \vec{\bar{g}}^{r+1} \cdots \vec{\bar{g}}^{2p}  } } - \Eb_{\Theta} \Eb_\xi \insquare{ \vec{\bar{g}}^{1} \cdots \vec{\bar{g}}^{2p}  }  } \\
    &\quad\quad \leq \sup_{\norm{F}_{\textnormal{Lip}} \leq 2pM^{2p -1}L} \abs{ \TT_{N,r}\insquare{F} - \QQ_{N}\insquare{F}  } \\
    &\quad\quad \leq 2p M^{2p - 1} L\, \textnormal{W}(\PP_N, \QQ_N),
\end{align*}
where the first inequality follows from Lemma \ref{lemma:||F_r||_Lip}, and where the second inequality follows from Lemma \ref{lemma:jointConverge_implies_twistedMarginalConverge(Wasserstein)}.
Altogether, we have that \eqref{eq:ProjResult_supOverBL_LM_2p_expansion} can be upper bounded as
\begin{align*}
    \Eb_{\Theta}  \insquare{ \inparen{ \inangle{ g(\Theta^\top \vec{x})  } - \Eb_{\xi}\insquare{ g\inparen{  \Theta^\top\inangle{\vec{x}} + \sqrt{\rho - q}\bxi }  }  }^{2p}  } &\leq 2p M^{2p - 1} L \textnormal{W}(\PP_N, \QQ_N) \sum_{1 \leq r \leq 2p} \binom{2p}{r},
\end{align*}
and \eqref{eq:projResult_partialAsymptotic} follows from the bound in Lemma \ref{thm:UnconditionalProjResult_X_PPNtoQQN}.
The proof for \eqref{eq:projResult_asymptotic} is analogous. We re-define $\vec{\bar{g}}^\ell := g\inparen{\sqrt{q}\vec{z} + \sqrt{\rho - q}\bxi^\ell}$ and repeat the arguments with the outer expectation over $\Theta$ and $\vec{z}$ and use Lemma \ref{lemma:PPN_to_QQ}.
\end{proofof}

\section{Examples}
\label{sec:examples}

\begin{example}[A random variant of classical CLT] Let $\vec{x} = (x_i)_{i \leq N}$ be a random vector with independent sub-gaussian coordinates and $\Eb x_i = \sqrt{q}$, $\Var x_i = \rho - q$ for $0 < q < \rho$. Then, in the notation of Theorem \ref{thm:ProjResult_DisorderedCase_supOverBL_LM_2p},
\begin{align}
    \sup_{\substack{\norm{g}_{\textnormal{Lip}} \leq 1 \\ \norm{g}_{\infty} \leq 1} } \Eb_{\Theta} \insquare{ \inparen{ \Eb_\vec{x}\insquare{g(\Theta^\top \vec{x})  } - \Eb_{\xi} \insquare{  g\inparen{\Theta^\top \Eb\insquare{\vec{x}} + \sqrt{\rho - q}\xi  }  }  }^{2p}  }  \leq \frac{K(p,\rho,q,\bar{C})k}{\sqrt{N}},
    \label{eq:randProj_toyExample}
\end{align}
where $\bar{C} := \max_i \norm{x_i}_{\psi_2}$, with sub-gaussian norm $\norm{y}_{\psi_2} := \inf\inbraces{ t > 0 : \Eb \exp(y^2/t^2) \leq 2  }$.

To see this, observe that $N^{-1}\norm{\vec{x}}^{2} - \rho = N^{-1}\sum_{i \leq N} (x_i^2 - \rho)$ is a sum of independent centered sub-exponential r.v.'s with $\norm{x_i^2 - \rho}_{\psi_1} \leq K \norm{x_i^2 }_{\psi_1} \leq K \norm{x_i }^{2}_{\psi_2} \leq K \bar{C}^2$, where $\norm{y}_{\psi_1} := \inf\inbraces{ t > 0 : \Eb \exp(\abs{y}/t) \leq 2  }$ denotes sub-exponential norm. Concentration results, e.g.~Bernstein's inequality (\cite{vershynin2018high} Corollary 2.8.3) yield that for $u \geq 0$,
\begin{align*}
    \PP\insquare{  \abs{\frac{1}{N} \norm{\vec{x}}^2 - \rho } \geq u } \leq 2\exp\inparen{-K N \min \inbraces{ \frac{u^2}{\bar{C}^4}, \frac{u}{\bar{C}^2}  }}.
\end{align*}
The above tail probability can be integrated to give $\Eb\insquare{ \inparen{ N^{-1}\norm{\vec{x}}^2 - \rho  }^2   } \leq K(\bar{C})/N$. Since the product of sub-gaussians is sub-exponential, an analogous argument will give $\Eb\insquare{ \inparen{ N^{-1}\vec{x}^1 \cdot \vec{x}^2 - q  }^2   } \leq K(\bar{C})/N$. Applying Theorem \ref{thm:ProjResult_DisorderedCase_supOverBL_LM_2p} with $c_1 = K(\bar{C})/N = c_2$ yields \eqref{eq:randProj_toyExample}.
\end{example}



\begin{example}[Joint law of cavity fields in Sherrington-Kirkpatrick (SK) model]
The SK Hamiltonian is the function $H_N : \inbraces{\pm 1}^N \rightarrow \RR$ defined by
\begin{align*}
    - H_N(\vec{x}) &= \frac{\beta}{\sqrt{N}}\sum_{i < j \leq N} g_{ij} x_i x_j + h\sum_{i \leq N} x_i,
\end{align*}
where $\inparen{g_{ij}}_{i < j}$ are independent standard Gaussians, $\beta \in [0, \infty)$ is the inverse temperature, and $h \in \RR$ defines an external field. The expectation over the disorder, the $g_{ij}$'s, is denoted by $\Eb_{\ud}$. The SK Gibbs measure $\inangle{\cdot}$, defined conditionally on $(g_{ij})$, is a (random) probability measure on $\inbraces{\pm 1}^N$ defined by, for integrable $f : \inbraces{\pm 1} \rightarrow \RR$,
\begin{align*}
    \inangle{f} = \frac{1}{Z_{N}(\beta, h)} \sum_{\vec{x} \in \inbraces{\pm 1}^N} f(\vec{x}) \exp(-H_N(\vec{x})),
\end{align*}
where $Z_{N}(\beta, h)$ is a normalizing constant. Observe that since $\vec{x} \in \inbraces{\pm 1}^N$, then $\norm{\vec{x}}^2/N \equiv 1$. At sufficiently high temperature, say $\beta < 1/2$, it is known that the overlap concentrates \cite[Equation (1.89)]{talagrand2010mean}: $\Eb_\ud \inangle{ \inparen{ N^{-1}\vec{x}^1 \boldsymbol{\cdot} \vec{x}^2 - q}^2 } \leq K/N$, where $q$ is the solution to $q = \Eb \tanh^2\inparen{ \beta\sqrt{q}z + h  }$, $z \sim \cN(0,1)$. These imply that the hypotheses \eqref{eq:projResult_TSOC_hypotheses_disordered} in Corollary \ref{thm:projResult_disordered} are satisfied with $\rho = 1$, $c_1 = 0$, $c_2 = K/N$.

In the cavity method (see \cite{mezard1987spin} Chapter V, \cite{talagrand2010mean} Section 1.6), it is of interest to compute the large-$N$ distribution of the cavity fields $\ell_i = \btheta_i^\top \vec{x}$, where $\vec{x}$ is drawn from $\inangle{\cdot}$, and $\btheta_i \overset{\textnormal{iid}}{\sim} \cN(0, N^{-1} I_N)$ is drawn independent of everything else, including the disorder. The cavity fields show up when we try to decouple a small number of spins (say $k = 2$) from the system, for purposes of computing free energy,  magnetizations, etc. For instance, consider an $(N+2)$-system with Hamiltonian $H_{N+2}^{+}$ with temperature parameter $\beta_+$ chosen such that $\beta_+/\sqrt{N + 2} = \beta/\sqrt{N}$. We can isolate the fields experienced by $x_{N+1}$ and $x_{N+2}$ as follows:
\begin{align*}
    -H_{N+2}^+(x_1,\dots,x_{N+2}) &= \frac{\beta_+}{\sqrt{N+2}} \sum_{i < j \leq N+2} g_{ij} x_i x_j + h \sum_{i \leq N + 2} x_i \\
    &= -H_{N}(\vec{x}) + x_{N+1} \inparen{  \frac{\beta}{\sqrt{N}} \sum_{i \leq N} g_{i, N+1} x_i + h  } + x_{N+2} \inparen{  \frac{\beta}{\sqrt{N}} \sum_{i \leq N} g_{i, N+2} x_i + h  } \\
    &=: -H_{N}(\vec{x}) + x_{N+1} \inparen{ \beta \ell_{N+1} + h } + x_{N+2} \inparen{ \beta \ell_{N+2} + h },
\end{align*}
where we have identified $\btheta_{N+j} := \inparen{ N^{-1/2} g_{N+j, i} }_{i \leq N}$, so that $\ell_{N + j} = \btheta_{N + j}^\top \vec{x}$. Note that $\vec{x}$ drawn according to the Gibbs measure $\inangle{\cdot}$ associated to $H_{N}$ is independent of the projection directions $\btheta_{N+j}$'s.

In general, we consider $k$ local fields $\ell_1,\dots,\ell_k$ with projection directions $\Theta = (\btheta_1, \dots, \btheta_k)$. An application of Corollary \ref{thm:projResult_disordered} gives
\begin{align}
    \sup_{\norm{g}_{\textnormal{BL}} \leq 1} \Eb_{\Theta} \Eb_{\ud} \insquare{ \inparen{ \inangle{ g\inparen{\ell_{1},\dots,\ell_{k}} } - \Eb_\xi \insquare{ g\inparen{ \btheta^\top\!\inangle{\vec{x}} + \sqrt{\rho - q} \bxi  } }  }^{2p}  } \leq \frac{K(p,q)}{\sqrt{N}}.
    \label{eq:example_SK_jointCavityFields}
\end{align}
When $k = 1$, that is when we consider only one cavity site, result \eqref{eq:example_SK_jointCavityFields} should be compared to the `central limit theorems for cavity fields' seen in the spin glass literature for the SK model or its $p$-spin variants \cite[Theorem 1.7.11]{talagrand2010mean}, \cite[Theorem 2]{chen2010central}, \cite[Theorem 1.5]{chatterjee2010spin}, \cite[Theorem 5.1]{chen2022tap}. The class of test functions differs; and the convergence rate in \eqref{eq:example_SK_jointCavityFields} is weaker than those results whose proof leverages SK specific properties.

However, the additional information provided by \eqref{eq:example_SK_jointCavityFields} is that when multiple cavity sites are considered, the cavity fields are not only approximately jointly Gaussian for large $N$, but also conditionally independent given the disorder. It is seen in \cite{wee2022local} that this leads to a statement about the asymptotic independence of any finite subset of coordinates in $\vec{x}$. Moreover, the general (not specific to any spin glass) approach towards random projections given in this paper means this asymptotic independence holds universally for many mean-field spin glasses in high temperature.

\end{example}

\section{Proof of partial converse}
\label{sec:partial_converse}

The below lemma converts conditional weak convergence to unconditional weak convergence of the joint replicated statistics. It can be thought of as a generalization of the `Hoeffding lemma' in \cite[Lemma 4.1]{dumbgen2013low} where it is required there that the limiting joint measure is a product measure. 

\begin{lemma}
    Suppose \eqref{eq:projResult_asymptotic_ifDirection} holds, then 
    \begin{align}
        \cL\!\inparen{ \begin{bmatrix}
            \Theta^\top \vec{x}^1 \\
            \Theta^\top \vec{x}^2
        \end{bmatrix} }  \longsquiggly \cL\!\inparen{ \begin{bmatrix}
            \sqrt{q}\vec{z} + \sqrt{\rho - q}\bxi^1 \\
            \sqrt{q}\vec{z} + \sqrt{\rho - q}\bxi^2
        \end{bmatrix} }
        \label{eq:projResult_asymptotic_joint_converges}
    \end{align}
    \label{lemma:hoeffding_onlyIf}
\end{lemma}

\begin{proof}
    Let $f,g : \RR^k \rightarrow \RR$ be continuous and bounded functions. Functions on $\RR^{2k}$ of the form $(x,y) \mapsto f(x) g(y)$ constitute a separating class of test functions for probability measures on $\RR^{2k}$ \cite[Lemma 1.4.2]{van1996weak}. We have
    \begin{align*}
        \Eb \insquare{ f(\Theta^\top \vec{x}^1) g(\Theta^\top \vec{x}^2) } &= \Eb_{\Theta} \insquare{ \inangle{f(\Theta^\top \vec{x}^1)} \inangle{g(\Theta^\top \vec{x}^2)} } \longrightarrow \Eb_{\vec{z}} \Eb_{\bxi} \insquare{ f(\sqrt{q}\vec{z} + \sqrt{\rho - q} \bxi^1 ) g(\sqrt{q}\vec{z} + \sqrt{\rho - q} \bxi^2 )  },
    \end{align*}
    where the convergence follows because if $A_N := \inangle{f(\Theta^\top \vec{x}^1)}$, $B_N := \inangle{g(\Theta^\top \vec{x}^2)}$, $A := \Eb_{\bxi} \insquare{ f(\sqrt{q}\vec{z} + \sqrt{\rho - q} \bxi^1 )   }$, and $B := \Eb_{\bxi} \insquare{ g(\sqrt{q}\vec{z} + \sqrt{\rho - q} \bxi^2 )   }$, then \eqref{eq:projResult_asymptotic_ifDirection} gives $A_N \overset{\textnormal{in prob.}}{\longrightarrow} A$ and $B_N \overset{\textnormal{in prob.}}{\longrightarrow} B$ and this leads to $A_N B_N \overset{\textnormal{in prob.}}{\longrightarrow} AB$ by the continuous mapping theorem. Furthermore, $A_N, B_N, A, B$ are uniformly bounded in $N$, which yields $\Eb \insquare{A_N B_N} \rightarrow \Eb \insquare{AB}$ as desired.  
\end{proof}

The proof of the partial converse is given next. It is largely similar to the strategy in \cite[Theorem 2.1]{dumbgen2013low}, with accommodations for the nonzero overlap. 

\begin{proofof}{Proposition \ref{prop:projResult_converse}} It suffices to prove \eqref{eq:projResult_converse_inProb}, since \eqref{eq:projResult_converse_inLr} follows from standard arguments to upgrade convergence in probability to convergence in $L_r$, under \eqref{eq:projResult_converse_UI_in_r_hypothesis}. Denote $\vec{Y}_N^{\ell} := \Theta^\top \vec{x}^\ell$ 
for $\ell = 1,2$. By Lemma \ref{lemma:hoeffding_onlyIf}, \eqref{eq:projResult_asymptotic_ifDirection} gives \eqref{eq:projResult_asymptotic_joint_converges}, which entails that the joint characteristic functions converge. That is, for all $\vec{t}_1 , \vec{t}_2 \in \RR^k$, we have 
\begin{align}
    \Eb \exp i \inparen{\vec{t}_1^\top \vec{Y}_N^1 + \vec{t}_1^\top \vec{Y}_N^2} \longrightarrow \exp \inparen{ -\frac{\norm{\vec{t}_1}^2}{2} \rho -\frac{\norm{\vec{t}_2}^2}{2} \rho - \vec{t}_1^\top \vec{t}_2 q  }.
    \label{eq:projResult_converse_jointCF_converges}
\end{align}
On the other hand, by conditioning on $\vec{x}^1, \vec{x}^2$ first, the LHS can be written as 
\begin{align}
    \Eb \exp i \inparen{\vec{t}_1^\top \vec{Y}_N^1 + \vec{t}_1^\top \vec{Y}_N^2} = \Eb \exp \inparen{ -\frac{\norm{\vec{t}_1}^2}{2} \frac{\norm{\vec{x}^1}^2}{N} -\frac{\norm{\vec{t}_2}^2}{2} \frac{\norm{\vec{x}^2}^2}{N}- \vec{t}_1^\top \vec{t}_2 \frac{\vec{x}^1\boldsymbol{\cdot} \vec{x}^2}{N}  }.
    \label{eq:projResult_converse_evaluate_finiteCF}
\end{align}
Setting $\vec{t}_2 = \vec{0}$ in \eqref{eq:projResult_converse_jointCF_converges} and \eqref{eq:projResult_converse_evaluate_finiteCF}, we obtain
\begin{align*}
    \Eb \exp \inparen{ - \lambda \frac{\norm{\vec{x}}^2}{N}  } \longrightarrow \exp \inparen{ - \lambda \rho  }, \quad \textnormal{for all } \lambda \geq 0,
\end{align*}
i.e.~the Laplace transform of the nonnegative r.v.~$N^{-1}\norm{\vec{x}^2}$ converges on $[0,\infty)$ to that of the constant r.v.~$\rho$. It follows that $\cL(N^{-1} \norm{\vec{x}}^2) \longsquiggly \delta_\rho$. We then use the fact that weak convergence to a point mass implies convergence in probability to that degenerate r.v.

Denote $a_1 = \norm{\vec{t}_1}^2/2$, $a_2 = \norm{\vec{t}_2}^2/2$, $a_3 = \vec{t}_1^\top \vec{t}_2$ and define
\begin{align*}
    H_N(a_1, a_2, a_3) := \Eb \exp \inparen{ -a_1 \frac{\norm{\vec{x}^1}^2}{N} -a_2 \frac{\norm{\vec{x}^2}^2}{N}- a_3 \frac{\vec{x}^1\boldsymbol{\cdot} \vec{x}^2}{N}  },
\end{align*}
so that $H_N$ is a reparametrization of the LHS of \eqref{eq:projResult_converse_jointCF_converges}. By choosing $\vec{t}_1$ such that $\norm{\vec{t}_1}^2/2 = 1$, and choosing $\vec{t_2} \in \inbraces{\vec{t_1}, - \vec{t_1}, \vec{0}}$, it is easily checked by \eqref{eq:projResult_converse_jointCF_converges} that
\begin{align}
    0 &= \lim_{N \rightarrow \infty} \insquare{e^{2q} H_N(1,1,2) + e^{-2q} H_N(1,1,-2) - 2 \inparen{H_N(1,0,0)}^2   } \nonumber\\
    &= \lim_{N \rightarrow \infty} 2 \cdot \Eb\insquare{ \exp \inparen{- \frac{\norm{\vec{x}^1}^2}{N} -  \frac{\norm{\vec{x}^2}^2}{N}  }  } \inparen{  \cosh\inparen{ \frac{2 \vec{x}^1 \boldsymbol{\cdot} \vec{x}^2}{N} - 2q  } - 1  }.
    \label{eq:projResult_converse_HN_linComb_limit}
\end{align}
Note that $\cosh \geq 1$, and that $\cosh$ is strictly increasing on $[0, \infty)$. For any $\epsilon > 0$ and $r > 0$ we have
\begin{align*}
    &\Eb\insquare{ \exp \inparen{- \frac{\norm{\vec{x}^1}^2}{N} -  \frac{\norm{\vec{x}^2}^2}{N}  }  } \inparen{  \cosh\inparen{ \frac{2 \vec{x}^1 \boldsymbol{\cdot} \vec{x}^2}{N} - 2q  } - 1  } \\
    &\quad \geq e^{-2r} (\cosh(2\epsilon) - 1) \cdot  \PP\insquare{ \frac{\norm{\vec{x}^1}^2}{N} < r, \frac{\norm{\vec{x}^2}^2}{N} < r, \abs{\frac{\vec{x}^1 \boldsymbol{\cdot} \vec{x}^2}{N} - q} \geq \epsilon  } \\
    &\quad \geq e^{-2r} (\cosh(2\epsilon) - 1) \inparen{ \PP\insquare{ \abs{\frac{\vec{x}^1 \boldsymbol{\cdot} \vec{x}^2}{N} - q} \geq \epsilon } - 2\cdot\PP\insquare{ \frac{\norm{\vec{x}}^2}{N} \geq r }  }.
\end{align*}
Sending $N \rightarrow \infty$ followed by $r \rightarrow \infty$, we obtain from \eqref{eq:projResult_converse_HN_linComb_limit} and monotone convergence theorem that 
\begin{align*}
    \limsup_{N \rightarrow \infty} \PP\insquare{ \abs{\frac{\vec{x}^1 \boldsymbol{\cdot} \vec{x}^2}{N} - q} \geq \epsilon } \leq 0
\end{align*}
which finishes the proof.
\end{proofof}

\appendix
\section{Supplementary proofs for Section \ref{sec:proof_main_result}}
\label{sec:suppProofs_for_randomProjs}

\subsection{Multivariate normal approximation by Stein's method with infinitesimal exchangeable pairs}
\label{sec:multivar_normal_SteinsMethod_infExchangeablePairs}


Let $\vec{w}$ be a random vector in $\RR^d$ which is conjectured to have a centered multivariate Gaussian distribution. That is, with $\bSigma$ a positive semidefinite matrix, and $\vec{Z}$ a standard Gaussian random vector in $\RR^d$, we wish to show $\cL(\vec{w}) \simeq \cL(\bSigma^{1/2} \vec{z})$. It is natural to compare expectations over a class of test functions, say $\cG = \inbraces{g : \RR^d \rightarrow \RR : \norm{g}_{\textnormal{Lip}} \leq 1}$. However, instead of working directly with $\abs{ \Eb g(\vec{w}) - \Eb g(\bSigma^{1/2}\vec{z}) }$ for some $g \in \cG$, Stein's method first solves the following differential equation, also called the `Stein characterizing equation':
\begin{align}
    \inangle{ \Hess f(\vec{w}), \bSigma  }_{\textnormal{HS}} - \inangle{\vec{w}, \nabla f(\vec{w})} = g(\vec{w}) - \Eb g(\bSigma^{1/2}\vec{z}).
    \label{eq:SteinCharEqn_multNormal_nonIdentityCovariance}
\end{align}
The solution is the `Stein transform' of $g$: $U_o g(\vec{w}) = -\int_{0}^{1} (2t)^{-1}\inparen{ \Eb g(\sqrt{t}\vec{w} + \sqrt{1 - t} \bSigma^{1/2}\vec{z}) - \Eb g(\bSigma^{1/2}\vec{z})  } \ud t$ \cite{barbour1990stein}, \cite{gotze1991rate}. Consequently, we shift our focus to bounding $\abs{ \Eb \inangle{ \Hess U_o g(\vec{w}), \bSigma  }_{\textnormal{HS}} - \inangle{\vec{w}, \nabla U_o(g)(\vec{w})}  }$. For this we will use the multivariate normal infinitesimal exchangeable pairs approach set forth by Meckes for the identity covariance $\bSigma = I$ case in her PhD thesis \cite{meckes2006PhD}, building upon earlier work by Stein \cite{stein1995accuracy}, and further developed in \cite{chatterjeeMeckes2007multivariate} and \cite{reinert2009multivariate}, the latter generalizing to non-identity covariance. The form that we use in Theorem \ref{thm:Meckes09_Thm4} comes from a subsequent paper \cite{meckes2009stein} which consolidates the existing results.

For a general introduction to Stein's method we refer the reader to \cite{ross2011fundamentals}, and for a treatment focused on normal approximation we refer to \cite{chen2011normal}.

In the computations that follow, we will need to evaluate moments of entries of Haar-distributed orthogonal matrices.

\begin{lemma}
\label{lemma:haar_orthMatrix_entryMoments}
Let $U = \inparen{u_{ij}}$ be an $N \times N$ matrix drawn from Haar measure on the orthogonal group. Then
\begin{enumerate}
    \item The entries of $U$ are identically distributed,
    \item $\Eb u_{11}$ = 0,
    \item $\Eb \insquare{u_{11}^2} = \frac{1}{N}$,
    \item $\Eb \insquare{u_{11}^2 u_{12}^2} = \frac{1}{N(N+2)}$,
    \item $\Eb \insquare{ u_{11}^2 u_{22}^2  } = \frac{N+1}{(N-1)N(N+2)}$,
    \item $\Eb u_{ij} u_{kl} u_{mn} u_{pq}$ is nonzero only if there is an even number of entries from each row and each column,
    \item $\Eb \insquare{ u_{11} u_{12} u_{21} u_{22}   } = \frac{-1}{(N-1)N(N+2)}$,
    \item for $i \neq k$, $j \neq \ell$, 
    \begin{align*}
        \Eb \insquare{ (u_{i1} u_{k2} - u_{i2} u_{k1}) (u_{j1} u_{\ell2} - u_{j2} u_{\ell 1}) } = \frac{2}{N(N-1)} \inparen{\delta_{ij} \delta_{k\ell} - \delta_{i\ell} \delta_{kj}}.
    \end{align*}
    \item Let $U^1 = (u^{1}_{ij})$, $U^2 = (u^{2}_{ij})$ be independently drawn. For $i \neq k$, $j \neq \ell$, 
    \begin{align*}
        \Eb \insquare{ (u^1_{i1} u^1_{k2} - u^1_{i2} u^1_{k1}) (u^2_{j1} u^2_{\ell2} - u^2_{j2} u^2_{\ell 1}) } = 0.
    \end{align*}
\end{enumerate}
\end{lemma}

\begin{proof}
Point 1 is given in \cite{meckes2019book} Lemma 2.1, while points 2 to 9 are given in \cite{meckes2006PhD} Lemma 3.3. It remains only to prove point 9, but this is straightforward. We have
\begin{align*}
    &\Eb \insquare{ (u^1_{i1} u^1_{k2} - u^1_{i2} u^1_{k1}) (u^2_{j1} u^2_{\ell2} - u^2_{j2} u^2_{\ell 1}) } \\
    &\quad\quad = \Eb\insquare{u^1_{i1} u^1_{k2}} \Eb\insquare{  u^2_{j1} u^2_{\ell 2}  } - \Eb\insquare{u^1_{i1} u^1_{k2}} \Eb\insquare{ u^2_{j2} u^2_{\ell 1}  } - \Eb\insquare{u^1_{i2} u^1_{k1}} \Eb\insquare{  u^2_{j1} u^2_{\ell 2}  } + \Eb\insquare{u^1_{i2} u^1_{k1}} \Eb\insquare{  u^2_{j2} u^2_{\ell 1}  } \\
    &\quad\quad = \Eb\insquare{u^1_{11} u^1_{22}} \Eb\insquare{  u^2_{11} u^2_{2 2}  } - \Eb\insquare{u^1_{11} u^1_{22}} \Eb\insquare{ u^2_{22} u^2_{1 1}  } - \Eb\insquare{u^1_{22} u^1_{11}} \Eb\insquare{  u^2_{11} u^2_{2 2}  } + \Eb\insquare{u^1_{22} u^1_{11}} \Eb\insquare{  u^2_{22} u^2_{1 1}  } \\
    &\quad\quad = 0,
\end{align*}
where the second line follows from the identical distribution of the entries (point 1).
\end{proof}

\begin{theorem}[Meckes \cite{meckes2009stein} Theorem 4, Reinert-R{\"o}llin \cite{reinert2009multivariate} Theorem 2.1] Let $Y$ be a random vector in $\RR^d$. For each $\epsilon \in (0,1)$ let $(Y, Y_\epsilon)$ be an exchangeable pair. Let $\cA$ be a sigma-algebra such that $\sigma(Y) \subseteq \cA$. Suppose there exists an invertible $d \times d$ matrix $\Lambda$, a symmetric, positive definite $d \times d$ matrix $\Sigma$, a random vector $E \in \RR^d$, measurable wrt.~$\cA$, a $d \times d$ random matrix $F$, measurable wrt.~$\cA$, and a deterministic function $s(\epsilon)$ such that, as $\epsilon \rightarrow 0$,
\begin{enumerate}
    \item $\frac{1}{s(\epsilon)} \Eb \insquare{Y_{\epsilon} - Y \, | \, \cA} \underset{\epsilon \rightarrow 0}{\xrightarrow{ \, L_1 \, }} - \Lambda Y + E$
    \item $\frac{1}{s(\epsilon)} \Eb \insquare{(Y_{\epsilon} - Y)(Y_{\epsilon} - Y)^\top \, | \, \cA}
    \underset{\epsilon \rightarrow 0}{\xrightarrow{ L_1(\norm{\cdot}_{\textnormal{HS}}) }} 2\Lambda \Sigma + F$
    \item For each $\delta > 0$, $\frac{1}{s(\epsilon)} \Eb \insquare{ \norm{Y_\epsilon - Y}^2 \boldsymbol{1}_{\inbraces{ \norm{Y_\epsilon - Y}^2 > \delta  } }  } \rightarrow 0 $.
\end{enumerate}
Then 
\begin{align*}
    \sup_{\norm{g}_{\textnormal{Lip}} \leq L} \abs{ \Eb g(Y) - \Eb g(\Sigma^{1/2} Z)  } \leq L \norm{\Lambda^{-1}}_{\textnormal{op}}  \inparen{  \Eb \norm{E}_{2} + \frac{1}{2} \norm{\Sigma^{-1/2}}_{\textnormal{op}} \Eb \norm{F}_{\textnormal{HS}}    },
\end{align*}
where $Z$ is a standard Gaussian random vector in $\RR^d$.
\label{thm:Meckes09_Thm4}
\end{theorem}

\begin{remark}
Observe that, should 
    \begin{align}
        \lim_{\epsilon \rightarrow 0} \frac{1}{s(\epsilon)} \Eb \norm{Y_\epsilon - Y}^3 = 0,
        \label{eq:meckes_cond3_sufficientCondToSatisfy}
    \end{align}
    then condition 3 holds. 
\end{remark}

\subsection{Proof of Lemma \ref{lemma:PPN_to_QQ}}

To prepare for the proof of Lemma \ref{lemma:PPN_to_QQ}, we make the following computation first. As $N$ grows large, we expect that the distribution of the projections $\Theta^\top \vec{x}^1, \dots, \Theta^\top \vec{x}^{2p}$ will be close to a $2kp$-Gaussian vector with zero mean and covariance $\Sigma \in \RR^{2kp \times 2kp}$ given by 
\begin{align}
    \Sigma &:= R^{2p}_{\rho, q} \otimes I_k = \begin{bmatrix}
    \rho I_k & q I_k & q I_k & \dots & q I_k \\
    q I_k & \rho I_k & q I_k & \dots & q I_k \\
    q I_k & q I_k & \rho I_k & \dots & q I_k \\
    \vdots & \vdots & \vdots & \ddots & \vdots \\
    q I_k & q I_k & q I_k & \dots & \rho I_k
    \end{bmatrix}.
    \label{eq:Sigma_covarianceOfUncenteredProjection}
\end{align}

\begin{lemma}
Let $\Sigma$ be defined as in \eqref{eq:Sigma_covarianceOfUncenteredProjection}. Then
\begin{align*}
    \norm{\Sigma^{-1/2}}_{\textnormal{op}} = \frac{1}{\sqrt{\rho - q}}.
\end{align*}
\label{lemma:||Sigma^-1/2||_op_Computation}
\end{lemma}

\begin{proofof}{Lemma \ref{lemma:||Sigma^-1/2||_op_Computation}}
It is possible to permute the rows and columns of $\Sigma$ such that it becomes a $k$-block-diagonal matrix with $2p \times 2p$ blocks $R^{2p}_{\rho, q}$. That is, there exists a permutation matrix $P \in \RR^{2kp \times 2kp}$ such that
\begin{align*}
    \Sigma = P \inparen{ \bigoplus_{i \leq k} R^{2p}_{\rho, q}  } P^\top.
\end{align*}
To find an inverse for $\Sigma$, it suffices to find an inverse for $R^{2p}_{\rho, q}$. By Sherman-Morrison formula applied in the last equality below, we have
\begin{align*}
    \inparen{ R^{2p}_{\rho, q} }^{-1} &= \inparen{ (\rho - q) I_{2p} + q \boldsymbol{1}\boldsymbol{1}^\top }^{-1} =: R^{2p}_{a,b},
\end{align*}
where 
\begin{align*}
    a &= \frac{1}{\rho - q} - \frac{1}{c(\rho - q)^2}; \quad b = - \frac{1}{c(\rho - q)^2},
\end{align*}
where $c = (1/q) + 2p/(\rho - q)$. By spectral decomposition, there exists orthonormal matrices $\Psi \in \RR^{2p \times 2p}$ such that 
\begin{align*}
    R^{2p}_{a,b} = \Psi^\top \diag\inparen{ a + (2p-1) b, a-b,a-b,\dots,a-b  } \Psi.
\end{align*}
It is easily checked that the eigenvalues above are all positive. Therefore
\begin{align*}
    \Sigma^{-1/2} &= P \inparen{ \bigoplus_{i \leq k} \Psi^\top \diag\inparen{ \sqrt{ a + (2p-1) b}, \sqrt{a-b},\sqrt{a-b},\dots,\sqrt{a-b}  } \Psi   }P^\top,
\end{align*}
and we obtain
\begin{align*}
    \norm{\Sigma^{-1/2}}_{\textnormal{op}} = \max\inparen{ \sqrt{a + (2p-1) b}, \sqrt{a - b}  }.
\end{align*}
The result follows upon substituting the values for $a$ and $b$.
\end{proofof}

\begin{proofof}{Lemma \ref{lemma:PPN_to_QQ}}
In this proof, it will be notationally convenient to write $\Eb$ for an expectation over all sources of randomness. This means, for instance, that $\Eb\insquare{\Theta^\top \vec{x}} = \Eb_{\Theta}\inangle{\Theta^\top \vec{x}}$ and also $\Eb \vec{x} = \inangle{\vec{x}}$. The strategy is to apply Theorem \ref{thm:Meckes09_Thm4}.

\textbf{\underline{Step 1:}} (Construction of exchangeable pair). Denote the columns of $\Theta$ by $(\btheta_1, \btheta_2, \dots, \btheta_k)$. An exchangeable pair for the $2kp$-vector
\begin{align}
    \inparen{ \Theta^\top \vec{x}^1, \dots, \Theta^\top \vec{x}^{2p}  } = \inparen{  \inparen{ (\btheta_j)^\top \vec{x}^\ell }_{j \leq k}  }_{\ell \leq 2p}
    \label{eq:projResult_proofmainResult_unconditionalJoint}
\end{align}
is constructed as follows (largely inspired by the strategy of \cite[Theorem 3.1]{stein1995accuracy} and \cite[Section 5.2]{meckes2006PhD}). Let $A_\epsilon$ be the $N \times N$ matrix given by
\begin{align*}
    A_{\epsilon} &:= \begin{bmatrix}
    \sqrt{1 - \epsilon^2} & \epsilon \\
    -\epsilon & \sqrt{1 - \epsilon^2}
    \end{bmatrix} \oplus I_{N-2} \\
    &= I_N + \inparen{ \epsilon \begin{bmatrix}
    0 & 1 \\
    -1 & 0
    \end{bmatrix} - \begin{bmatrix}
    1 - \sqrt{1 - \epsilon^2} & 0 \\
    0 & 1 - \sqrt{1 - \epsilon^2}
    \end{bmatrix}} \oplus 0_{N-2} \\
    &= I_N + \inparen{\epsilon C_2 - \inparen{\frac{\epsilon^2}{2} + O(\epsilon^4) }I_2} \oplus 0_{N-2}.
\end{align*}
where it is used that $1 - \sqrt{1 - \epsilon^2} = \epsilon^2/2 + O(\epsilon^4)$, as $\epsilon \rightarrow 0$, and where
\begin{align*}
    C_2 &= \begin{bmatrix}
    0 & 1 \\
    -1 & 0
    \end{bmatrix}.
\end{align*}
The action of $A_{\epsilon}$ is to rotate clockwise by $\arcsin(\epsilon)$ in the first two coordinates. The computation for $A_{\epsilon}$ shows that it can be written as a perturbation. Next let $U_1, \dots, U_k$ be $k$ independent random $N \times N$ orthogonal matrices drawn from the Haar measure on the orthogonal group, independent of all other sources of randomness. Define the vector
\begin{align*}
    \btheta_{j}^{\epsilon} := U_j A_{\epsilon} U_j^\top \btheta_j, \quad j \leq k,
\end{align*}
which can be thought of as a rotation of $\btheta_j$ in a random two-dimensional subspace. Let $\Theta_\epsilon = (\btheta_1^\epsilon, \btheta_2^\epsilon, \dots, \btheta_k^\epsilon)$ and set the exchangeable pair to \eqref{eq:projResult_proofmainResult_unconditionalJoint} to be $\inparen{ \Theta_\epsilon^\top \vec{x}^1, \dots, \Theta_\epsilon^\top \vec{x}^{2p}  }$.

\textbf{\underline{Step 2:}} (Verify condition 1 of Theorem \ref{thm:Meckes09_Thm4}).
Let $K_j \in \RR^{N \times 2}$ be the first two columns of $U_j$. The following computation is used many times in the sequel: for each $j$,
\begin{align}
    \btheta_j^{\epsilon} - \btheta_j &= U_j \inparen{ I_N + \inparen{\epsilon C_2 - \inparen{\frac{\epsilon^2}{2} + O(\epsilon^4) }I_2} \oplus 0_{N-2} } U_j^\top \btheta_j - \btheta_j \nonumber\\
    &= K_j \inparen{ \epsilon C_2  - \inparen{\frac{\epsilon^2}{2} + O(\epsilon^4) }I_2  } K_j^\top \btheta_j.
    \label{eq:theta_eps_minus_theta}
\end{align}
By Lemma \ref{lemma:haar_orthMatrix_entryMoments} it is straightforward to compute (see also Meckes thesis \cite{meckes2006PhD} proof of Theorem 3.1),
\begin{align*}
    \Eb K_j K_j^\top &= \frac{2}{N}I_N; \quad\quad \Eb K_j C_2 K_j^\top = 0_{N \times N}.
\end{align*}
We now compute the matrix $\Lambda$ in Theorem \ref{thm:Meckes09_Thm4} using \eqref{eq:theta_eps_minus_theta}:
\begin{align*}
    \frac{N}{\epsilon^2} \Eb \insquare{ \left.  \begin{bmatrix}
    \Theta_{\epsilon}^\top \vec{x}^1  \\
    \vdots\\
    \Theta_{\epsilon}^\top \vec{x}^{2p} 
    \end{bmatrix} - \begin{bmatrix}
    \Theta^\top \vec{x}^{1} \\
    \vdots\\
    \Theta^\top \vec{x}^{2p}
    \end{bmatrix} \right\rvert \begin{bmatrix}
    \Theta^\top \vec{x}^1 \\
    \vdots \\
    \Theta^\top \vec{x}^{2p}
    \end{bmatrix} } &= \Eb \insquare{ \left. \begin{bmatrix}
    \vrule\\
    \frac{N}{\epsilon} (0_{N \times N} \btheta_j)^\top \vec{x}^{\ell} - \frac{N}{2} \inparen{ \frac{2}{N} I_N \btheta_j}^\top \vec{x}^{\ell} \\
    \vrule
    \end{bmatrix} \; \right\rvert \begin{bmatrix}
    \Theta^\top \vec{x}^{1} \\
    \vdots \\
    \Theta^\top \vec{x}^{2p}
    \end{bmatrix} } + \begin{bmatrix}
    \vrule \\
    O(\epsilon) \\
    \vrule
    \end{bmatrix} \\
    &= - \begin{bmatrix}
    \Theta^\top \vec{x}^{1} \\
    \vdots\\
    \Theta^\top \vec{x}^{2p}
    \end{bmatrix} + \begin{bmatrix}
        \vrule \\
        O(\epsilon) \\
        \vrule
    \end{bmatrix},
\end{align*}
where the quantity $O(\epsilon)$ may depend on $N$, $\vec{x}^1, \dots, \vec{x}^{2p}$, $\Theta$. The above implies that the matrix $\Lambda$ in Theorem \ref{thm:Meckes09_Thm4} can be taken as $\frac{1}{N}I_{2kp}$. Moreover, the random vector $E$ in Theorem \ref{thm:Meckes09_Thm4} condition 1 is the zero vector. 

\textbf{\underline{Step 3:}} (Verify condition 3 of Theorem \ref{thm:Meckes09_Thm4}). Condition 3 in Theorem \ref{thm:Meckes09_Thm4} is satisfied (as in \eqref{eq:meckes_cond3_sufficientCondToSatisfy})---the calculation in \eqref{eq:theta_eps_minus_theta} gives
\begin{align*}
    \Eb \insquare{ \left. \norm{ \begin{bmatrix}
    \Theta_{\epsilon}^\top \vec{x}^{1} \\
    \vdots\\
    \Theta_{\epsilon}^\top \vec{x}^{2p}
    \end{bmatrix} - \begin{bmatrix}
    \Theta^\top \vec{x}^{1} \\
    \vdots\\
    \Theta^\top \vec{x}^{2p} 
    \end{bmatrix} }^3 \right\rvert \begin{bmatrix}
    \Theta^\top \vec{x}^1 \\
    \vdots\\
    \Theta^\top \vec{x}^{2p}
    \end{bmatrix} }  = O(\epsilon^3).
\end{align*}

\textbf{\underline{Step 4:}} (Verify condition 2 of Theorem \ref{thm:Meckes09_Thm4}).
To compute the random matrix $F$ in Theorem \ref{thm:Meckes09_Thm4} condition 2, we first evaluate the following $2kp \times 2kp$ random matrix
\begin{align}
    \tilde{F} &:= \Eb \insquare{ \left.  \inparen{\begin{bmatrix}
    \Theta_{\epsilon}^\top \vec{x}^{1}  \\
    \vdots\\
    \Theta_{\epsilon}^\top \vec{x}^{2p} 
    \end{bmatrix} - \begin{bmatrix}
    \Theta^\top \vec{x}^{1} \\
    \vdots\\
    \Theta^\top \vec{x}^{2p}
    \end{bmatrix} } \inparen{\begin{bmatrix}
    \Theta_{\epsilon}^\top \vec{x}^{1}  \\
    \vdots\\
    \Theta_{\epsilon}^\top \vec{x}^{2p} 
    \end{bmatrix} - \begin{bmatrix}
    \Theta^\top \vec{x}^{1} \\
    \vdots\\
    \Theta^\top \vec{x}^{2p}
    \end{bmatrix} }^\top \right\rvert \begin{bmatrix}
    \Theta^\top \vec{x}^{1} \\
    \vdots\\
    \Theta^\top \vec{x}^{2p}
    \end{bmatrix} } = \begin{bmatrix}
    \tilde{M}_{11} & \dots & \tilde{M}_{1,2p} \\
    \vdots &  \ddots & \vdots \\
    \tilde{M}_{2p,1} & \dots & \tilde{M}_{2p, 2p}
    \end{bmatrix},
    \label{eq:randomMatrix_F_main}
\end{align}
where $\tilde{M}_{\ell,\ell'}$, $1 \leq \ell, \ell' \leq 2p$ are $k \times k$ blocks given by
\begin{align*}
    \tilde{M}_{\ell,\ell'} &:=  \Eb \insquare{ \left.  (\Theta_{\epsilon} - \Theta)^\top \vec{x}^{\ell} (\vec{x}^{\ell'})^\top(\Theta_{\epsilon} - \Theta)  \right\rvert \Theta^\top \vec{x}^\ell, \Theta^\top \vec{x}^{\ell'} }.
\end{align*}
It will be seen next that each of $\epsilon^{-2}\tilde{M}_{\ell, \ell'}$ can be expressed as the sum of diagonal random matrices and a perturbation of order $O(\epsilon) \boldsymbol{1}\boldsymbol{1}^\top$. First we make the following observations, which are used many times in the sequel: for $1 \leq j, j' \leq k$, using \eqref{eq:theta_eps_minus_theta} 
\begin{align}
    &(\btheta_j^{\epsilon} - \btheta_j)(\btheta_{j'}^{\epsilon} - \btheta_{j'})^\top \nonumber\\
    &\quad\quad = \inparen{ K_j \inparen{ \epsilon C_2  - \inparen{\frac{\epsilon^2}{2} + O(\epsilon^4) }I_2  } K_j^\top \btheta_j  } \inparen{K_{j'} \inparen{ \epsilon C_2  - \inparen{\frac{\epsilon^2}{2} + O(\epsilon^4) }I_2  } K_{j'}^\top \btheta_j}^\top \nonumber\\
    &\quad\quad = \epsilon^2 \inparen{ K_{j} C_2 K_j^\top \btheta_j }\inparen{ K_{j'} C_2 K_{j'}^\top \btheta_{j'}}^\top -2\inparen{ \frac{\epsilon^3}{2} + O(\epsilon^5)  } \inparen{ K_j C_2 K_{j'}^\top\btheta_j} \inparen{ K_{j'}C_2 K_{j'}^\top\btheta_{j'}}^\top \nonumber\\
    &\quad\quad\quad\quad + \inparen{\frac{\epsilon^2}{2} + O(\epsilon^4)}^2 \inparen{K_{j}K_{j}^\top \btheta_j} \inparen{K_{j'}K_{j'}^\top \btheta_{j'}}^\top.
    \label{eq:(theta_eps_minus_theta)(theta_eps_minus_theta)^T}
\end{align}
The last two terms are of higher order than $\epsilon^2$, and in what follows they can be absorbed into an $O(\epsilon)$ term. Furthermore, with $U = (u_{ij})$ and $K$ the first two columns of $U$, the $(i,j)$-th entry of $K C_2 K^\top$ is given by
\begin{align*}
    (K C_2 K^\top)_{ij} &= u_{i1}u_{j2} - u_{i2}u_{j1}.
\end{align*}
We return to the computation in \eqref{eq:randomMatrix_F_main}. The diagonal $k \times k$ blocks $\tilde{M}_{1,1},\dots \tilde{M}_{2p,2p}$ have the following diagonal entries. It suffices to compute the $(1,1)$-entry in $\tilde{M}_{11}$ and by symmetry the rest will follow. To ease notation we suppress the indices so that $\vec{x}^{1} \equiv \vec{x}$, $\btheta_1 \equiv \btheta$, and $\btheta_1^\epsilon \equiv \btheta^\epsilon = U A_\epsilon U^\top \btheta$. By \eqref{eq:(theta_eps_minus_theta)(theta_eps_minus_theta)^T},
\begin{align}
    &\inparen{\tilde{M}_{1,1}}_{1,1} = \frac{1}{\epsilon^2} \Eb \insquare{  \vec{x}^\top (\btheta^{\epsilon} - \btheta)(\btheta^{\epsilon} - \btheta)^\top \vec{x} \; | \; \Theta^\top \vec{x} } \nonumber\\
    &\quad\quad = \Eb \insquare{   \vec{x}^\top \inparen{ K C_2 K^\top \btheta }  \inparen{ K C_2 K^\top \btheta }^\top \vec{x}  \; | \; \btheta^\top \vec{x} } + O(\epsilon) \nonumber\\
    &\quad\quad = \Eb \insquare{ \sum_{i,j \leq N} x_i x_j \Eb\insquare{ (K C_2 K^\top \btheta)_i (K C_2 K^\top \btheta)_j   \; | \; \vec{x}, \btheta  } \; \bigg\rvert \; \btheta^\top \vec{x}  } + O(\epsilon) \nonumber\\
    &\quad\quad = \Eb \insquare{ \sum_{i,j \leq N} x_i x_j \Eb\insquare{ \inparen{\sum_{k \leq N} (u_{i1} u_{k2} - u_{i2} u_{k1}) \theta_{k} } \inparen{\sum_{\ell \leq N} (u_{j1}u_{\ell 2} - u_{j2}u_{\ell 1}) \theta_\ell }   \; \bigg\rvert \; \vec{x}, \btheta  } \; \bigg\rvert \; \btheta^\top \vec{x}  } + O(\epsilon) \nonumber\\
    &\quad\quad = \Eb \insquare{ \sum_{i,j \leq N} \sum_{\substack{k \neq i \\ \ell \neq j}} x_i x_j \theta_k \theta_\ell  \Eb \insquare{ (u_{i1}u_{k2} - u_{i2}u_{k1}) (u_{j1}u_{\ell 2} - u_{j2}u_{\ell 1})  }  \; \Bigg\rvert \; \btheta^\top \vec{x}  } + O(\epsilon) 
    \label{eq:randomMatrixF_common_expansion}
\end{align}
By Lemma \ref{lemma:haar_orthMatrix_entryMoments} point 8, we further have
\begin{align}
    &\inparen{\tilde{M}_{1,1}}_{1,1} = \frac{2}{N(N-1)} \Eb \insquare{  \sum_{i,j \leq N} \sum_{\substack{k \neq i \\ \ell \neq j}} x_i x_j \theta_k \theta_\ell \inparen{ \delta_{ij} \delta_{k\ell} - \delta_{i\ell}\delta_{jk} }  \; \Bigg\rvert \; \btheta^\top \vec{x}  } + O(\epsilon) \nonumber\\
    &\quad\quad = \frac{2}{N(N-1)} \Eb \insquare{  \sum_{i \neq k} x_{i}^{2} \theta_k^{2} - \sum_{i \neq k} x_i x_k \theta_i \theta_k  \; \Bigg\rvert \; \btheta^\top \vec{x}  } + O(\epsilon) \nonumber\\
    &\quad\quad = \frac{2}{N(N-1)} \Eb \insquare{  \sum_{i \leq N} x_{i}^{2}\inparen{\sum_{k \leq N} \theta_k^{2} - \theta^{2}_{i}} - \sum_{i \neq k} x_i x_k \theta_i \theta_k  \; \Bigg\rvert \; \btheta^\top \vec{x}  } + O(\epsilon) \nonumber\\
    &\quad\quad = \frac{2}{N(N-1)} \Eb \insquare{ \norm{\vec{x}}^2\norm{\btheta}^2 - (\btheta^\top \vec{x})^2   \; \bigg\rvert \; \btheta^\top X  } + O(\epsilon).
    \label{eq:randomMatrixF_tildeM_1,1_1,1_finalExpression}
\end{align}
On the other hand, the diagonal $k \times k$ blocks in \eqref{eq:randomMatrix_F_main} have the following off-diagonal entries (again it suffices to compute the $(1,2)$-entry of $\tilde{M}_{11}$ by symmetry). Similar to \eqref{eq:randomMatrixF_common_expansion}, we have (again with $\vec{x}^1 \equiv \vec{x}$)
\begin{align*}
    &\frac{1}{\epsilon^2} \Eb \insquare{  \vec{x}^\top (\btheta_1^{\epsilon} - \btheta_1)(\btheta_2^{\epsilon} - \btheta_2)^\top \vec{x} \; | \; \Theta^\top \vec{x} } \\
    &\quad = \Eb \insquare{ \sum_{i,j \leq N} \sum_{\substack{k \neq i \\ \ell \neq j}} x_i x_j \theta_{1,k} \theta_{2,\ell}  \Eb \insquare{ (u_{1,i1}u_{1,k2} - u_{1,i2}u_{1,k1}) (u_{2,j1}u_{2,\ell 2} - u_{2,j2}u_{2,\ell 1})  }  \; \Bigg\rvert \; \btheta_1^\top \vec{x}, \btheta_2^\top \vec{x}  } + O(\epsilon) \\
    &\quad = O(\epsilon),
\end{align*}
which follows from Lemma \ref{lemma:haar_orthMatrix_entryMoments} point 9. This completes the computation of entries in the diagonal $k\times k$ blocks $\tilde{M}_{11}, \dots, \tilde{M}_{2p,2p}$ in \eqref{eq:randomMatrix_F_main}.

The off-diagonal $k \times k$ blocks $\tilde{M}_{\ell, \ell'}$, $\ell \neq \ell'$ in \eqref{eq:randomMatrix_F_main} are computed similarly. The $(1,1)$-entry in $\tilde{M}_{1,2}$ is given as follows. By similar steps that lead to \eqref{eq:randomMatrixF_tildeM_1,1_1,1_finalExpression},
\begin{align*}
    &\frac{1}{\epsilon^2} \Eb \insquare{  (\vec{x}^{1})^\top (\btheta_1^{\epsilon} - \btheta_1)(\btheta_1^{\epsilon} - \btheta_1)^\top \vec{x}^{2} \; | \; \Theta^\top \vec{x}^1, \Theta^\top \vec{x}^{2} } \\
    &\quad\quad = \frac{2}{N(N-1)} \Eb \insquare{ (\vec{x}^1)^\top \vec{x}^2 \norm{\btheta_1}^2 - (\btheta_1^\top \vec{x}^1)(\btheta_1^\top \vec{x}^2)   \; \bigg\rvert \; \Theta^\top \vec{x}^1, \Theta^\top \vec{x}^{2}  } + O(\epsilon).
\end{align*}
Furthermore, the off-diagonal entries in $\tilde{M}_{1,2}$ are of order $O(\epsilon)$ by similar reasoning. Altogether, since $F = \lim_{\epsilon \rightarrow 0} \epsilon^{-2}\tilde{F} - 2 \Lambda \Sigma$, where $\Sigma = R^{2p}_{\rho, q} \otimes I_k$, we obtain that $F$ in Theorem \ref{thm:Meckes09_Thm4} condition 2 can be set as
\begin{align*}
    F &:= \begin{bmatrix}
    M_{11} & \dots & M_{1,2p} \\
    \vdots &  \ddots & \vdots \\
    M_{2p,1} & \dots & M_{2p, 2p}
    \end{bmatrix},
\end{align*}
where $M_{\ell, \ell'}$, $1 \leq \ell, \ell' \leq 2p$, are $k \times k$ random diagonal matrices given by
\begin{align*}
    M_{\ell, \ell'} &= \Eb \insquare{ \left. \begin{bmatrix}
    \ddots & & \\
    & \frac{2}{N(N-1)} \inparen{ (\vec{x}^{\ell})^\top \vec{x}^{\ell'} \norm{\btheta_j}^2 - (\btheta_j^\top \vec{x}^{\ell})(\btheta_j^\top \vec{x}^{\ell'})  } - \frac{2}{N} \inparen{ \rho \boldsymbol{1}_{\ell = \ell'} + q \boldsymbol{1}_{\ell \neq \ell'}  } & \\
    & & \ddots
    \end{bmatrix} \; \right\rvert \begin{bmatrix} \btheta^\top \vec{x}^{\ell} \\ \btheta^\top \vec{x}^{\ell'} \end{bmatrix} }
\end{align*}
where the $(j,j)$-th elements are indicated, for $1 \leq j \leq k$.

Since $\Eb \norm{F}_{\textnormal{HS}} \leq \Eb \norm{F}_{L_1}$, it remains to bound the expected magnitude of the entries of $F$.  Denoting $\btheta \equiv \btheta_1$, and $\vec{x} \equiv \vec{x}^{1}$, define the quantities
\begin{align*}
    \textnormal{I} &:=  \sqrt{  \Eb \insquare{  \inparen{  \norm{\vec{x}}^2\norm{\btheta}^2 - N\rho  } ^{2} } } \\
    \textnormal{II} &:= \sqrt{  \Eb \insquare{  \inparen{  (\vec{x}^{1})^\top \vec{x}^2 \norm{\btheta}^2 - Nq  } ^{2} } } \\
    \textnormal{III} &:= \Eb \insquare{(\btheta^\top \vec{x})^2}.
\end{align*}
For the diagonal terms in $F$, it again suffices by symmetry to consider the (1,1)-entry of $M_{11}$: $F_{11}$. We have
\begin{align*}
    \Eb \abs{F_{11}} &= \frac{2}{N(N-1)} \Eb \abs{ \Eb \insquare{  \norm{\vec{x}^1}^2\norm{\btheta}^2 - (\btheta^\top \vec{x})^2 - (N-1)\rho  \big\rvert \btheta^\top \vec{x}^1 } }\\
    &\leq \frac{2}{N(N-1)} \inparen{ \textnormal{I}  + \textnormal{III} + \rho  },
\end{align*}
which follows from Jensen's inequality and tower property of expectations.

On the other hand, to deduce a bound for all off-diagonal entries of $F$, it suffices by symmetry to compute the $(1,1)$-entry of $M_{1,2}$: $F_{1,k+1}$. Similarly, we have
\begin{align*}
    \Eb \abs{F_{1,k+1}} &= \frac{2}{N(N-1)} \Eb \abs{ \Eb \insquare{  (\vec{x}^1)^\top\vec{x}^2 \norm{\btheta}^2 - (\btheta^\top \vec{x}^1)(\btheta^\top \vec{x}^2) \; \big\rvert \; (\btheta^\top \vec{x}^1, \btheta^\top \vec{x}^2) } } \\
    &\leq \frac{2}{N(N-1)} \inparen{ \textnormal{II} + \textnormal{III} + q },
\end{align*}
where we have also used Cauchy-Schwarz and symmetry among replicas to produce term $\textnormal{III}$.

There are $2kp$ diagonal entries in $F$, and $2k\cdot\binom{2p}{2}$ non-zero off-diagonal terms in $F$. Furthermore, from Lemma \ref{lemma:||Sigma^-1/2||_op_Computation} $\norm{\Sigma^{-1/2}}_{\textnormal{op}} = \frac{1}{\sqrt{\rho - q}}$, and $\norm{\Lambda^{-1}}_{\textnormal{op}} = N$, so that Theorem \ref{thm:Meckes09_Thm4} yields
\begin{align}
    \sup_{\norm{g}_{\textnormal{Lip}} \leq 1} \abs{ \PP_{N}\insquare{g} - \QQ\insquare{g} } &\leq \frac{8p^2 k }{\sqrt{\rho - q} (N-1)} \inparen{ \textnormal{I}  + \textnormal{II} + \textnormal{III} + \rho + q }
    \label{eq:Xdist_Proj2p_inTermsOf_I_II_III}
\end{align}
It remains to upper bound $\textnormal{I}$, $\textnormal{II}$, $\textnormal{III}$. A straightforward computation, using that $\Eb \norm{\btheta}^2 = 1$, gives the identity
\begin{align*}
    \textnormal{I}^2 &= \Eb \insquare{ \norm{\btheta}^4 } \Eb \insquare{  \inparen{ \norm{\vec{x}}^2 - N\rho }^2   } + N \rho \Var\insquare{\norm{\btheta}^2}  \inparen{ 2\Eb \insquare{\norm{\vec{x}}^2}  - N \rho  }.
\end{align*}
Writing $2\Eb \insquare{\norm{\vec{x}}^2}  - N \rho \leq 2 \sqrt{\Eb \insquare{\inparen{ \norm{\vec{x}}^2 - N\rho }^2}} + N\rho$, and using that $\Var \norm{\btheta}^2 = 2/N$ so that $\Eb \norm{\btheta}^4 \leq 3$, we obtain from the given hypotheses that
\begin{align*}
    \textnormal{I} &\leq \sqrt{3N^2 c_1 + 4N\rho \sqrt{c_1} + 2N\rho^2 } = d_1(c_1).
\end{align*}
An analogous computation yields that
\begin{align*}
    \textnormal{II} &\leq \sqrt{3N^2 c_2 + 4Nq\sqrt{c_2} + 2Nq^2 } = d_2(c_2).
\end{align*}
On the other hand, since $\Eb \insquare{  \btheta \btheta^T } = N^{-1} I_{N}$, we have
\begin{align*}
    \textnormal{III} &= \frac{1}{N} \Eb \norm{\vec{x}}^2 = \Eb \insquare{ \frac{\norm{\vec{x}}^2}{N} - \rho }  + \rho  \leq \sqrt{c_1} + \rho.
\end{align*}
Substituting these bounds into \eqref{eq:Xdist_Proj2p_inTermsOf_I_II_III} finishes the proof.
\end{proofof}

\begin{proofof}{Lemma \ref{lemma:||F_r||_Lip}}
The proof is by induction on $r$. The case $r = 1$ is immediate. For $r \geq 2$, let $x_1, \dots, x_r, y_1, \dots, y_r \in \RR^{k}$. We have
\begin{align*}
    &\abs{ F_r(x_1,\dots,x_r) - F_r(y_1,\dots,y_r)  } \\
    &\quad\quad \leq \abs{ g(x_r) \inparen{ F_{r-1}(x_1,\dots,x_{r-1}) - F_{r-1}(y_1,\dots, y_{r-1})  } } + \abs{ \inparen{g(x_r) - g(y_{r}) } F_{r-1}(y_1, \dots, y_{r-1})  } \\
    &\quad\quad \leq M \norm{F_{r-1}}_{\textnormal{Lip}} + M^{r-1} \norm{g}_{\textnormal{Lip}},
\end{align*}
and the result follows from the induction hypothesis.
\end{proofof}

\begin{lemma} Let $\PP$ and $\QQ$ be probability measures on $\RR^M$. For any subset $I \subseteq \inbraces{1,\dots,M}$, let $\PP_I$ and $\QQ_I$ be the marginals of $\PP$ and $\QQ$ on those coordinates indexed by $I$. Then
\begin{align*}
    \textnormal{W}(\PP_I, \, \QQ_I) &\leq \textnormal{W}\!\inparen{ \PP, \, \QQ }.
\end{align*}
\label{lemma:jointConverge_implies_marginalConverge(Wasserstein)}
\end{lemma}

\begin{proofof}{Lemma \ref{lemma:jointConverge_implies_marginalConverge(Wasserstein)}}
Let $\mu \equiv \cL_{\mu}\inparen{ \inparen{X_i}_{i \leq M}, \, \inparen{Y_i}_{i \leq M}  }$ be an almost optimal coupling of $\PP$ and $\QQ$ for the Wasserstein infimum, i.e., for $\epsilon > 0$,
\begin{align}
    \int  \norm{ \inparen{ (X_i)_{i \leq M} -  (Y_i)_{i \leq M}  } }  \, \ud \mu &\leq \textnormal{W}\!\inparen{ \PP, \, \QQ  } + \epsilon,
    \label{eq:jointConverge_implies_marginalConverge(Wasserstein)_OPTCOUPLING}
\end{align}
where the distance on the LHS corresponds to Euclidean distance on $\RR^{M}$. Let
\begin{align*}
    \mu_I := \cL_{\mu}\inparen{(X_i)_{i \in I}, \, (Y_i)_{i \in I}},
\end{align*}
be the corresponding marginal of $\mu$ on coordinates in $I$. 
Therefore
\begin{align*}
    \textnormal{W}\!\inparen{  \PP_I, \, \QQ_I }  
    &\leq \int \norm{ (X_i)_{i \in I} - (Y_i)_{i \in I} } \; \ud \mu_I  = \int \norm{ (X_i)_{i \in I} - (Y_i)_{i \in I} } \; \ud \mu \leq \int \norm{ (X_i)_{i \leq M} - (Y_i)_{i \leq M} } \; \ud \mu,
\end{align*}
and the result follows from \eqref{eq:jointConverge_implies_marginalConverge(Wasserstein)_OPTCOUPLING}.
\end{proofof}

\begin{proofof}{Lemma \ref{lemma:jointConverge_implies_twistedMarginalConverge(Wasserstein)}}
Let $\mu$ be a near optimal coupling of $\PP$ and $\QQ$, i.e.~$\mu$ is the probability measure on $\RR^{2dD}$ with $\textnormal{W}(\PP, \QQ) \leq \int \norm{ \inparen{ y_{i} }_{i \leq dD} - \inparen{z_i}_{i \leq dD} }  \mu\inparen{\ud y_1, \dots, \ud y_{dD}, \ud z_1, \dots, \ud z_{dD}   } + \epsilon$, for some $\epsilon > 0$. Let $\mu_r$ be the marginal on the coordinates $f(Y^1, U), \dots, f(Y^r, U)$ and $f(Z^1, U), \dots, f(Z^r, U)$, that is, if $\pi : \RR^{2dD} \rightarrow \RR^{2dr}$ is the map $\pi\inparen{ \omega_1,\dots,\omega_{2dD}   } = \inparen{ \omega_{1},\dots, \omega_{dr}, \omega_{dD + 1}, \dots, \omega_{dD + dr} }$, then 
\begin{align*}
    \mu_r := \mu \circ  \pi^{-1}.
\end{align*}
Let $\nu = \cL\inparen{  f(Y^1, U),\dots, f(Y^r, U), f(Z^{r+1}, U), \dots, f(Z^D, U), \inparen{  f(Z^\ell, U)  }_{\ell \leq D}  }$ be the coupling between $\TT_r$ and $\QQ$ satisfying
\begin{align*}
    \nu \circ \pi^{-1} &= \mu_r,
\end{align*}
that is, the marginal on the coordinates $f(Y^1, U), \dots, f(Y^r, U)$ and $f(Z^1, U), \dots, f(Z^r, U)$ coincides with $\mu_r$. Then
\begin{align*}
    \textnormal{W}(\TT_r, \QQ) &\leq \int \norm{ \inparen{ y_{i} }_{i \leq dD} - \inparen{z_i}_{i \leq dD} }_{\RR^{dD}}  \nu\inparen{\ud y_1, \dots, \ud y_{dD}, \ud z_1, \dots, \ud z_{dD}   } \\
    &= \int \norm{ \inparen{ y_{i} }_{i \leq dr} - \inparen{z_i}_{i \leq dr} }_{\RR^{dr}}  \mu_r \inparen{\ud y_1, \dots, \ud y_{dr}, \ud z_1, \dots, \ud z_{dr}   } \\
    &\leq \textnormal{W}(\PP, \QQ),
\end{align*}
where the last inequality follows from Lemma \ref{lemma:jointConverge_implies_marginalConverge(Wasserstein)}.   
\end{proofof}

\addcontentsline{toc}{section}{References}
\bibliographystyle{alpha}
\bibliography{mybib.bib}

\end{document}